\documentclass[10pt]{amsart}
\usepackage{amsmath,amssymb}
\usepackage{amsthm}
\usepackage{amscd} 
\usepackage{graphicx} 
\usepackage{epstopdf}
\newtheorem{LM}{Lemma}[section]
\newtheorem{THM}[LM]{Theorem}
\newtheorem{Q}[LM]{Question}
\newtheorem{PR}[LM]{Proposition}
\newtheorem{COR}[LM]{Corollary}
\newtheorem{CL}[LM]{Claim}

\newtheorem{RK}[LM]{Remark}

\def\Vec#1{\overrightarrow{#1}}
\def\dps{\displaystyle}
\def\QHS{$\mathbb{Q}HS$}
\def\ZHS{$\mathbb{Z}HS$}
\def\C0{$\mathcal{C}^0$}
\def\c2{$\mathcal{C}^2$}
\def\1n{\hbox{$\{1,\dots,n\}$}}

\begin{document}
\title{Existence of Taut foliations on Seifert fibered homology $3$-spheres} 
\author{Shanti Caillat-Gibert}
\email{shanti@cmi.univ-mrs.fr}
\address{CMI, Universit\'e de Marseille-Provence,
13453-Marseille cedex 13, France}
\author{Daniel Matignon}
\email{matignon@cmi.univ-mrs.fr}
\address{CMI, Universit\'e de Marseille-Provence,
13453-Marseille cedex 13, France}

\date{\today}


\begin{abstract}
This paper concerns the problem of existence of taut foliations among $3$-manifolds.
Since the contribution of David Gabai \cite{Ga},
we know that closed $3$-manifolds with non-trivial second homology group
admit a taut foliations.
The essential part of this paper focuses on Seifert fibered homology $3$-spheres.
The result is quite different if they are integral or rational but non-integral homology $3$-spheres.
Concerning integral homology $3$-spheres, we prove that all but the $3$-sphere and the Poincar\'e $3$-sphere admit a taut foliation.
Concerning non-integral homology $3$-spheres,
we prove there are infinitely many which admit a taut foliation, and infinitely many without taut foliation.
Moreover, we show that the geometries do not determine the existence of taut foliations
on non-integral Seifert fibered homology $3$-spheres.
\end{abstract}

\maketitle

\tableofcontents

\section{Introduction}
All $3$-manifolds are considered compact, connected and orientable.
Taut foliations provide deep information on $3$-manifolds and their contribution in understanding the topology and geometry of $3$-manifolds is still in progress. The first result came from 
S. P. Novikov \cite{No} in 1965, who proved that a $3$-manifold which admit a taut foliation has to be irreducible or $\mathbb{S}^2\times\mathbb{S}^1$.
Since then, we know \cite{Pa} that such manifolds have $\mathbb{R}^3$ for universal cover,
and that their fundamental group is infinite \cite{No} and Gromov negatively curved when the manifold is also toroidal~\cite{Ca}.
Recently, W. P. Thurston has exhibit an approach with taut foliations towards the geometrization.

In \cite{Ga}, D. Gabai proved that a closed $3$-manifold with a non-trivial second homo\-logy group admits a taut foliation.
A lot of great works then concern the existence of taut foliations, see for examples \cite{B1, B2, BNR, Clauss, RSS}.
This paper seeks to answer the question for Seifert fibered $3$-manifolds.
In the following, a {\it non-integral homology $3$-sphere} means a rational
homology $3$-sphere, which is not an integral homo\-logy $3$-sphere.
The results are quite different if they are integral homology $3$-spheres,
or non-integral homology $3$-spheres.

\begin{THM}[Main Theorem 1]\label{integral thm}
Let $M$ be a Seifert fibered integral homology $3$-sphere.
Then $M$ admits a taut analytic foliation if and only if
$M$ is neither ho\-me\-omorphic to the $3$-sphere nor the Poincar\'e sphere.
\end{THM}

Concerning non-integral homology $3$-spheres, the non-existence is not isolated.
Of course, the $3$-sphere and lens spaces do not admit a taut foliation,
but for any choice of the number of exceptional slopes,
there exist infinitely many which admit a taut foliation, and infinitely many which do not.

\begin{THM}[Main Theorem 2]\label{rational thm}
Let $n$ be a positive integer greater than two.
Let $\mathcal{S}_n$ be the set of Seifert fibered $3$-manifolds whith $n$ exceptional fibers,
which are non-integral homology $3$-spheres.
For each $n$~:
\begin{itemize}
\item[(i)]
There exist infinitely many Seifert fibered manifolds in $\mathcal{S}_n$ which admit a taut analytic foliation; and
\item[(ii)]
There exist infinitely many Seifert fibered manifolds in $\mathcal{S}_n$ which do not admit a taut \c2-foliation.
\item[(iii)] 
There exist infinitely many Seifert fibered manifolds in $\mathcal{S}_3$ which do not admit a taut \C0-foliation.
\end{itemize}
\end{THM}

Actually, by considering the normalized Seifert invariant
$(0; b_0, b_1/a_1, \dots, b_n/a_n)$ of a Seifert fibered homology $3$-sphere,
and assuming that $b_0\not=-1$ (nor $1-n$),
then $b_0$ determines wether $M$ does or does not admit a taut \c2-foliation,
see Theo\-rem~\ref{ehn}, which collects results in \cite{EHN, JN, N}.
Note that there is a fiber-preserving homeo\-morphism of $M$ which switches $b_0=1-n$ to  $b_0=-1$.
Therefore, the problem remains open only for $b_0=-1$.
We will prove (see Theorem~\ref{geom_taut rational}) that even if
the $3$-manifolds all are equipped with $b_0=-1$,
Main Theorem~$2$ is still true.
To prove the non-existence of taut \c2-foliations, we first prove that 
a taut \c2-foliation can be isotoped to a horizontal one, and then use a characterization of 
horizontal foliations for Seifert fibered homology $3$-spheres
(see bellow for more details~: schedule of the paper).
So, the following result play a key-rule in the proof.

\begin{THM}\label{eq_taut_hor}
Let $M$ be a Seifert fibered rational homology $3$-sphere.
Let $n$ be the number of exceptional fibers of $M$.
If $n>3$ (resp. $n=3$) then
any taut \c2-foliation (resp. \C0-foliation)
of $M$ can be isotoped to be a horizontal foliation.
\end{THM}

Moreover, we will show that the geometries do not determine the existence of taut foliations
on Seifert fibered rational homology $3$-spheres.

\begin{THM}\label{geom_notaut}
Let $M$ be a Seifert fibered rational homology $3$-sphere.
If $M$ does not admit the
$\widetilde{SL}_2(\mathbb{R})$-geometry,
then $M$ does not admit a taut \c2-foliation.
\end{THM}

\begin{RK}
There exist infinitely many such manifolds (see Section~$7$)
but the converse is not true as says Theorem~\ref{geom_taut rational}~: we can give
infinitely many such manifolds, which admit the $\widetilde{SL}_2(\mathbb{R})$-geometry 
(and with $b_0=-1$) but no taut \c2-foliation.
\end{RK}

\begin{THM}\label{geom_taut integral}
Let $M$ be a Seifert fibered integral homology $3$-sphere.
If $M$ admits the $\widetilde{SL}_2(\mathbb{R})$-geometry, then
$M$ is neither homeomorphic to the $3$-sphere nor the Poincar\'e sphere.

In particular (Main Theorem~$1$)
$M$ admits a taut analytic foliation.
\end{THM}

\medskip
\noindent
{\sc Schedule of the paper.}
We organize the paper as follows.
\\\\
In Section $2$,
we recall basic definitions and notations on Seifert fibered $3$-manifolds,
taut or horizontal foliations and well known results.
\\\\
Section $3$ is devoted to the proof of Theorem~\ref{eq_taut_hor},
which is based on Proposition~\ref{sep_torus}, which claims that
a transversely oriented and taut foliation of a closed $3$-manifold cannot
contain a separating compact leaf.
Then, a taut \c2-foliation of a Seifert fibered homology $3$-sphere 
cannot contain a compact leaf (see Corollary~\ref{cor cpct leaf}).
Therefore, it can be isotoped to be horizontal (see Theorem~\ref{cpt_iso}),
by collecting the works on foliations \cite{B2, EHN, Le, Ma, No, Th} of
M.~Brittenham, 
D.~Eisenbud, 
U.~Hirsch, 
G.~Levitt, 
S.~Matsumoto, 
W.~Neumann, 
S.~P.~Novikov and W.~P.~Thurston.

Since a horizontal foliation is clearly a taut foliation,
an immediate consequence is that 
a Seifert fibered rational homology $3$-sphere,
$M$ say, admits a taut \c2-foliation 
if and only if $M$ admits a horizontal foliation (Corollary~\ref{equiv cor}).
This corollary was also proved by combining results 
\cite{ET, JN, LM, LS, N, OS}
of
Y.~Eliashberg,
M.~Jankins,
P.~Lisca,
G.~Mati\'c,
R.~Naimi,
W.~Neumann,
P. ~Ozsv\'ath, 
A.~I.~Stipsciz,
Z.~Szab\'o 
and W.~P.~Thurston
(for more details, see the end of Section~$3$).
\\\\
The goal of Section $4$ is a characterization of Seifert fibered rational
homology $3$-spheres, which admit a taut \c2-foliation.
Since a taut \c2-foliation can be isotoped to be horizontal,
we use the characterization \cite{JN,N} of 
M. Jankins, R.~Naimi and W. Neumann for horizontal foliations
(for more details, see Section~$4$).
This characterization gives rise to criteria to be satisfied by the Seifert invatiants.
\\\\
Section $5$ concerns the 
geometries of homology $3$-spheres.
We will prove the following result.
\begin{PR}\label{global}
Let $M$ be a Seifert fibered rational homology $3$-sphere, with $n$ exceptional fibers.
If $M$ does not admit the $\widetilde{SL}_2(\mathbb{R})$-geometry, then
the following statements all are satisfied.
\begin{itemize}
\item[(i)]
$n\leq 4$.
\item[(ii)] If $n=4$ then $M$ admits the $\mathcal{N}$il-geometry, and is a
non-integral homology $3$-sphere.
\item[(iii)]
If $M$ is an integral homology $3$-sphere,
then $M$ admits the $\mathbb{S}^3$-geometry and is either
homeomorphic to the $3$-sphere or to the Poincar\'e sphere.
\end{itemize}
\end{PR}
We may note that if $n=2$ then $M$ is a lens space (including $\mathbb{S}^3$ and $\mathbb{S}^1\times\mathbb{S}^2$).

We combine Proposition~\ref{global} with the criteria given by the characterization of Section~$4$,
to prove Theorem \ref{geom_notaut}.
\\\\
Section $6,\ 7$ and $8$ are devoted respectively to the proof of Theorem \ref{geom_notaut},
Theorem \ref{geom_taut rational} and Main Theorem $1$.

To prove Theorem \ref{geom_taut rational}, we first exhibit infinite families of
Seifert fibered  non-integral homology spheres,
which admit the $\widetilde{SL}_2(\mathbb{R})$-geometry (and $b_0=1$).
Then, we prove that they do satisfy (or do not satisfy)
the criteria of the characterization described in Section~$4$.

To prove Main Theorem $1$, we need to study more deeply these criteria.
\\\\
{\sc Perspectives.}
\\
By F. Waldhausen, \cite{Wa} we know that an
incompressible compact surface in a 
Seifert  fibered $3$-manifold (not necessarily a homology $3$-sphere)
 can be isotoped to be either horizontal or vertical.
This is clearly not the same for foliations.

A vertical leaf is homeomorphic to either a $2$-cylinder (${\mathbb S}^1\times\mathbb R$)
or a $2$-torus (${\mathbb S}^1\times{\mathbb S}^1$). Therefore, taut foliations are not necessary 
isotopic to vertical ones; and vice-versa, vertical foliations are not necessary 
isotopic to taut foliations, e.g. cylinders which wrap around two tori in a turbulization way;
for more details, see~\cite{Gi2}.
But clearly, horizontal foliations are taut.

By Theorem~\ref{cpt_iso},
a taut \c2-foliation can be isotoped to a horizontal foliation,
if there is no compact leaf.

We wonder if a taut \C0-foliation, without compact leaf, of a Seifert fibered $3$-manifold
can be isotoped to be horizontal and so analytic.
By \cite{BNR}, there exist manifolds which admit taut \C0-foliation but not taut \c2-foliation.
Therefore,
that seems impossible in general, but the question is still open for homology $3$-spheres.

\begin{Q}
Let $\mathcal{F}$ be a taut \C0-foliation, without compact leaf, of a Seifert fibered 
homology $3$-sphere.
Can $\mathcal{F}$ be isotoped to be horizontal~?
\end{Q}

M.~Brittenham~\cite{B2}, answers the question when the base is $\mathbb{S}^2$ with $3$ exceptional fibers,
see Remark~\ref{history} for more details.

Gluing Seifert fibered $3$-manifolds with boundary components along some of them (or all)
give {\it graph manifolds}.
We wonder if we can classify graph manifolds without taut foliations, with their Seifert fibered pieces
and gluing homeomorphims.

\begin{Q}
Let $M$ be a graph $3$-manifold.
What kind of obstructions are there for $M$ not to admit a taut foliation~?
\end{Q}

\section{Preliminaries}\label{prel}
We may recall here, that all $3$-manifolds are considered compact, connected and orientable.
This section is devoted to recall basic definitions and notations on Seifert fibered $3$-manifolds,
taut or horizontal foliations and well known results.

\medskip\noindent
{\bf Notations}
Let $M$ be a $3$-manifold.
If $M$ is an integral homology sphere, resp. a rational homology sphere,
we say that $M$ is a \ZHS, resp. a \QHS.
Clearly, a \ZHS\ is a \QHS.
If $M$ is a  \ZHS, resp. a \QHS, and a Seifert fibered $3$-manifold,
we say that $M$ is {\it a \ZHS, resp. a \QHS, Seifert fibered $3$-manifold}.

\medskip\noindent
{\bf Separating surfaces and non-separating surfaces.}
A properly embedded surface $F$ in a $3$-manifold $M$ is  said to be 
{\it a separating surface} if $M-F$ is not connected;
otherwise, $F$ is said to be {\it a non-separating surface} in $M$.
If $F$ is a separating surface, we call {\it the sides of $F$} the connected components of $M-F$.
Note that if $M$ is a \QHS\ manifold, then $M$ does not contain any non-separating surface.

A $3$-manifold is said to be {\it reducible} if $M$ contains an {\it essential $2$-sphere},
i.e. a $2$-sphere which does not bound any $3$-ball in $M$.
Then, either $M$ is homeomorphic to $\mathbb{S}^1\times\mathbb{S}^2$, or $M$ is a non-trivial connected sum.
If $M$ is not a reducible $3$-manifold, we say that $M$ is an {\it irreducible} $3$-manifold.
We may note that all Seifert fibered $3$-manifolds but
$\mathbb{S}^1\times\mathbb{S}^2$ and $\mathbb{R}P^3\#\mathbb{R}P^3$ are irreducible $3$-manifolds.

\medskip\noindent
{\bf Seifert fibered $3$-manifolds.}
We can find the first definition of Seifert fibered $3$-manifolds, called {\it fibered spaces} 
by H.~Seifert, in \cite{ST}. We first consider fibered solid tori.

The standard solid torus $V$ is said to be {\it $p/q$-fibered},
if $V$ is foliated by circles, such that the core is a leaf, and all the other leaves are circles
isotopic to the $(p,q)$-torus knot
(i.e. they run $p$ times in the meridional direction and 
$q$ times in the longitudinal direction) where $q\not=0$.
A solid torus $W$ is {\it $\mathbb{S}^1$-fibered} if $W$ is foliated by circles,
such that there exists a homeomorphism between $W$ and 
the $p/q$-fibered standard solid torus $V$, which preserves the leaves.
We may say that $W$ is a {\it $p/q$-fibered solid torus}.

A $3$-manifold $M$ is said to be {\it a Seifert fibered} $3$-manifold,
or {\it a Seifert fiber space}
if
$M$ is a disjoint union of simple circles, called {\it the fibers},
such that the regular neighborhood of each fiber is a $\mathbb{S}^1$-fibered solid torus.
Let $W$ be a $p/q$-fibered solid torus. If $q=1$, we say that its core is {\it a regular fiber};
otherwise we say that its core is {\it an exceptional fiber} and $q$ is 
{\it the multiplicity} of the exceptional fiber.

By D.~B.~A.~Epstein \cite{Ep} this is equivalent to say that 
$M$ is a $\mathbb{S}^1$-bundle over a $2$-orbifold. 

\medskip\noindent
{\bf Seifert invariants.}
In \cite{Se}, H. Seifert developped numerical invariants, which give a complete classification
of Seifert fibered $3$-manifolds.
Let $M$ be a closed Seifert manifold based on an orientable surface of genus $g$, with $n$ exceptional fibers. Let $V_1,\dots,V_n$ be the solid tori, which are
regular neighborhood of each exceptional fiber.
We do not need to consider non-orientable base surface here.
If we remove these solid tori,
we obtain a trivial $\mathbb{S}^1$-bundle over a genus $g$ compact surface,
whose boundary is a union of $2$-tori $T_1,\dots,T_n$;
where $T_i=\partial V_i$, for $i\in\{1,\dots,n\}$.
Gluing back $V_1,\dots,V_n$ consists to
assign to each of them a slope $b_i/a_i$~:
we glue $V_i$ along $T_i$, such that the slope $b_i/a_i$ on $V_i$ bounds a meridian disk of $V_i$.
Formally, if $f$ and $s$ represent respectively a fiber and a section on $T_i$, then
the boundary of the meridian disk of $V_i$ is attached along the slope represented by 
$a_i[s]+b_i[f]$ in $H_1(T_i,\mathbb{Z})$.

Clearly, $a_i\geq 2$ is the multiplicity of the core of $V_i$,
and $b_i$ depends on the choice of a section.
Removing the regular neighborhood of a regular fiber, we obtain 
an integer slope $b_0$.
Then, $g,b_0,b_1/a_1,\dots, b_n/a_n$
completely describe $M$. We denote $M$ by
$M(g; b_0, b_1/a_1, \dots, b_n/a_n)$,
which is called {\it the Seifert invariant}.

\medskip \noindent
{\bf Seifert normalized invariant and convention.}
New sections are obtained by Dehn twistings along the fiber (along annuli or tori);
therefore a new section does not change $b_i$ modulo $a_i$.
Thus, we can fix $b_0$
so that $0<b_i<a_i$ for $i\in\{1,\dots,n\}$.

That gives rise to the {\it Seifert normalized invariant}~:
$M(g; b_0, b_1/a_1, \dots, b_n/a_n)$;
i.e. $0<b_i<a_i$ for $i\in\{1,\dots,n\}$.

H. Seifert \cite{Se} showed that 
$M(g; b_0, b_1/a_1, \dots, b_n/a_n)$ is
fiber-preserving homeomorphic to
$-M(g; -n-b_0, 1-b_1/a_1, \dots, 1-b_n/a_n)$
where $-M$ denotes $M$ with the opposite orientation.
In all the following, we denote by $\Phi$ 
this isomorphism.
Therefore, we may assume that $b_0<0$
otherwise we switch for $-n-b_0$.
For more details, see \cite{Se} or \cite{BNR, Ha}.

Every \QHS\ Seifert fibered $3$-manifold $M$ is based on $\mathbb{S}^2$.
Indeed, every non-separating curve on the base surface
induces a non-separating torus in $M$; which cannot be in a \QHS.
Hence, the base surface of a \QHS\ Seifert fibered $3$-manifold is a $2$-sphere.

From now on,
we denote for convenience such $M$ by $M(-b_0, b_1/a_1, \dots, b_n/a_n)$,
where $b_0>0$ and $0<b_i<a_i$ for $i\in\{1,\dots,n\}$. We will write~:
$$M=M(-b_0, b_1/a_1, \dots, b_n/a_n).$$

\medskip \noindent
{\bf Euler number.}
When $M$ has a unique fibration, we denote by $e(M)$
{\it the Euler number of its fibration}.
Note that few Seifert $3$-manifolds (lens spaces and a finite number of others)
do not have a unique fibration, see \cite{Ha} for more details;
all of them but lens spaces and $\mathbb{S}^3$, are not homology $3$-spheres.
$$\dps e(M)=-b_0+\sum_{i=1}^{n} b_i/a_i.$$

\medskip \noindent
{\bf Taut foliations.} Let $M$ be a $3$-manifold and $\mathcal{F}$ a foliation of $M$.
A simple closed curve $\gamma$
(respectively, a properly embedded simple arc, when $\partial M\not=\emptyset$)
is called {\it a transverse loop}
(respectively a {\it transverse arc})
if $\gamma$ is transverse to $\mathcal{F}$,
i.e. $\gamma$ is transverse to every leaf $F\in\mathcal{F}$, such that $\gamma\cap F\not=\emptyset$.

We say that a foliation $\mathcal{F}$ is \textit{taut},
if for every leaf of $F$ of $\mathcal{F}$, there exists 
a transverse loop, or a transverse arc if $\partial M\not=\emptyset$, $\gamma$ say,
such that $\gamma\cap F\not=\emptyset$.

We end this part by the famous theorem of Gabai \cite{Ga} on the existence of taut foliations,
which is stated here for closed $3$-manifolds.

\begin{THM}[D. Gabai, \cite{Ga}]
Let $M$ be a closed $3$-manifold.
If $H_2(M; \mathbb{Q})$ is non-trivial then $M$ admits a taut foliation.
\end{THM}

\medskip \noindent
{\bf Horizontal and vertical foliations.}
Let $M$ be a Seifert fibered $3$-manifold and $\mathcal{F}$ a foliation of $M$.
We say that $\mathcal{F}$ is {\it horizontal} if each
$\mathbb{S}^1$-fiber is a transverse loop to $\mathcal{F}$.
We say that $\mathcal{F}$ is {\it vertical} if each leaf of $\mathcal{F}$ is $\mathbb{S}^1$-fibered,
i.e. a disjoint union of $\mathbb{S}^1$-fibers.

Note that
only Seifert fibered $3$-manifolds are concerned by horizontal or vertical foliations.
Horizontal foliations are sometimes just called \textit{transverse foliations} to underline the fact that 
horizontal foliations are transverse to the $\mathbb{S}^1$-fibers.

Clearly, horizontal foliations are taut, because any transverse fiber 
(meeting a leaf) is the required transverse loop; so we have the following result.
\begin{LM}\label{easy}
A horizontal foliation is taut.
\end{LM}
\section{Horizontal and taut \c2-foliations in Seifert fibered homology $3$-spheres}
This section is devoted to the proof of Theorem~\ref{eq_taut_hor};
then with Lemma~\ref{easy}, we obtain~:
\begin{COR}\label{equiv cor}
Let $M$ be a \QHS\ Seifert fibered $3$-manifold.
Let $n$ be the number of exceptional fibers of $M$.
If $n>3$ (resp. $n=3$)
then, $M$ admits a horizontal foliation if and only if $M$ admits a taut \c2-foliation
(resp. a \C0-foliation).
\end{COR}
There exists an alternative proof  (but not direct) of this corollary; see at the end of this section.
\noindent
\\\\
{\sc Proof of Theorem~\ref{eq_taut_hor}.}
\\
In the light of known results on foliations \cite{B2, EHN, Le, Ma, No, Th} of
M.~Brittenham, 
D.~Eisenbud, 
U.~Hirsch, 
G.~Levitt, 
S.~Matsumoto, 
W.~Neumann, 
S.~P.~Novikov and W.~P.~Thurston
(where Theorem~\ref{cpt_iso} is their collection)
it is sufficient to see that any taut foliation on a $\mathbb{Q}$HS Seifert fibered $3$-manifold,
has no compact leaf.
Then the result follows by Corollary~\ref{cor cpct leaf},
which claims that no leaf in 
a taut foliation of a \QHS\ can be compact.
\qed
\\

Corollary~\ref{cor cpct leaf} is an immediate consequence of Proposition~\ref{sep_torus},
which concerns all (compact, oriented and connected) closed $3$-manifolds,
and can be generalized with some boundary conditions to $3$-manifolds with 
non-empty boundary, see \cite{Gi2}.

 \begin{PR}\label{sep_torus}
A transversely oriented and taut foliation of a closed $3$-manifold,
cannot contain a compact separating leaf.
\end{PR}
For example, foliations of $3$-manifolds which admit a Reeb's component are not taut.

A taut foliation $\mathcal{F}$ is said to be {\it transversely oriented} if
there exists a one-dimensional oriented foliation $\mathcal{G}$ transverse to $\mathcal{F}$.
This is equivalent to say that the normal vector field to the tangent planes to the leaves
of $\mathcal{F}$ is continuous (and nowhere vanishes);
which gives the following consequence.
\begin{LM}\label{transv or}
Let $\mathcal{F}$  be a transversely oriented foliation.
If $F$ is a separating compact leaf then the continuous normal vector field to the tangent planes to $F$, has
all his vectors pointing the same side of $F$.
\end{LM}
Actually, only this property is used in the proof of Proposition~\ref{sep_torus}.
We wonder if taut foliations are transversely oriented, and vice-versa.
In fact, there exist taut foliations which are not transversely oriented,
see \cite{Gi2} for more details. The inverse is easy to construct, e.g a Reeb's component.
We may note that there exist also foliations  without non-orientable compact leaves,
which are neither taut nor transversely oriented.
\begin{LM}\label{taut imply transv or if}
Let $\mathcal{F}$ be a taut foliation.
If $\mathcal{F}$ does not contain a non-orientable surface
then $\mathcal{F}$ is a transversely oriented foliation.
\end{LM}
\begin{proof}
Let $\mathcal{F}$ be a foliation.
If all the leaves are orientable then 
we can consider a well defined normal vector field to the tangent planes of $\mathcal{F}$.
Moreover,
if the foliation is taut then we can choose a well defined and continuous vector field
(and nowhere vanishing).
Therefore, $\mathcal{F}$ is a transversely oriented foliation.
\end{proof}
Since a \QHS\ cannot contain a non-orientable surface neither a non-separating surface,
Proposition~\ref{sep_torus} gives immediately the following result.
\begin{COR}\label{cor cpct leaf}
A taut foliation of a \QHS\ cannot admit a compact leaf.
\end{COR}

\medskip\noindent
{\it Proof of Proposition~\ref{sep_torus}}
\\
Let $M$ be a closed $3$-manifold.
Let $\mathcal{F}$ be a taut foliation.
We proceed by contradiction. So, we assume that
$\mathcal{F}$ contains a compact separating leaf, $F$ say.
We will see that $\mathcal{F}$ cannot be taut; which is the required contradiction.
We follow here a M. Brittenham's argument in his notes \cite{B1} concerning the non-tatness of a Reeb's component.

Since $F$ is separating, $M-F$ contains two components.
Let $M_1$ be the closure of one component, and 
$\mathcal{F}_1= \mathcal{F}\cap M_1$.
Let $\gamma$ be a properly embedded arc in $M_1$.
We will show that $\gamma$ cannot be transverse to $\mathcal{F}_1$;
therefore $\mathcal{F}_1$ cannot be taut and $\mathcal{F}$ neither.

\begin{figure}[htbp]
\begin{center}
\includegraphics[width=.9\linewidth]{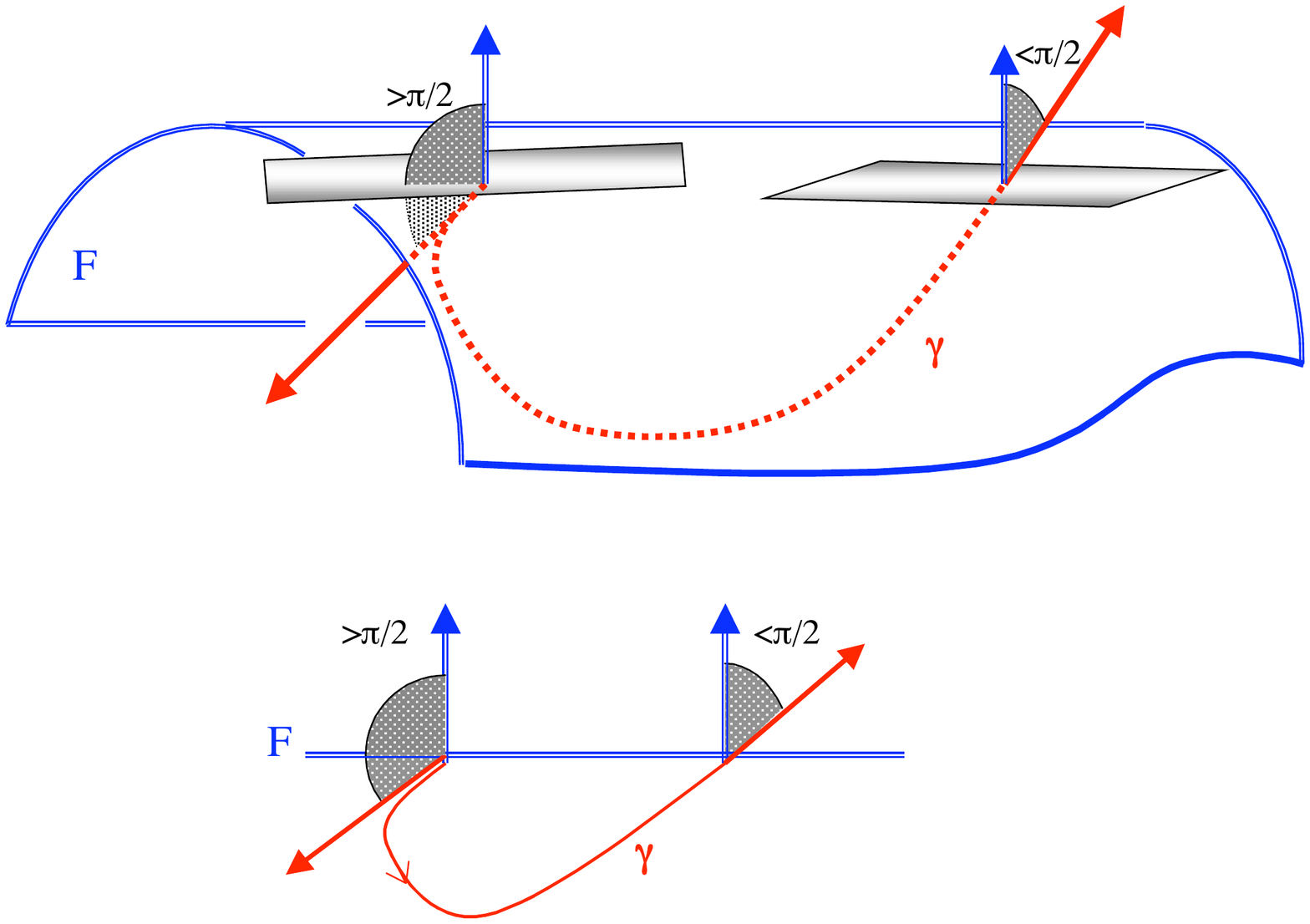}
\caption{}
\label{fig1}
\end{center}
\end{figure}

Note that $\gamma$ is a continuous map from $[0,1]$ to $M_1$.
We consider the following map $\dps h~: [0,1]\rightarrow  [-\frac{\pi}2,\frac{\pi}2]$ defined as follows~:

\[
\left.
\begin{array}{cccl}
h~: \hfill& [0,1] & \rightarrow & \dps  [-\frac{\pi}2,\frac{\pi}2] \hfill  \\
\\
 & t & \mapsto &\dps h(t) =
 \widehat{\big(\Vec{v_{\gamma(t)}},\Vec{n_{\gamma(t)}}\big)}-\frac{\pi}2
 \end{array}
\right.
\]

where $\widehat{\big(\Vec{v_{\gamma(t)}},\Vec{n_{\gamma(t)}}\big)}$ is 
the non-oriented angle between the tangent vector to $\gamma(t)$,
denoted by $\Vec{v_{\gamma(t)}}$,
and the normal vector to the tangent plane to the leaf of $\gamma(t)$,
denoted by $\Vec{n_{\gamma(t)}}$.
Since $M$ is a Rieman $3$-manifold, $h$ is trivially well defined.
Since the vector fields are continuous, $h$ is a continuous map.

We may note that $\widehat{\big(\Vec{v_{\gamma(t)}},\Vec{n_{\gamma(t)}}\big)}\in[0,\pi]$ by definition.
Note also that $h(t_0)=0$ if and only if $\gamma$ is tangent to the leaf at $\gamma(t_0)$.
We will prove that there exists such a $t_0$; so $\gamma$ is not 
transverse to $\mathcal{F}_1$.
 \\
Since $h$ is a continuous map,
by the Intermediate Value Theorem, it is sufficient to see that
$h(0)h(1)<0$; i.e. if $h(0)<0$ then $h(1)>0$ and vice-versa.
\\
We may assume that $h(0)h(1)\not=0$; otherwise $t_0\in\{0,1\}$.
Since $F$ is oriented, and $\mathcal{F}$ is transversely oriented, 
by Lemma~\ref{transv or}, the normal vectors
$\Vec{n_{\gamma(0)}}$ and $\Vec{n_{\gamma(1)}}$ both point outside $M_1$,
or both point inside $M_1$ (and $n_{\gamma(t)}$ never vanishes).
\\
On the other hand,
since $\Vec{v_{\gamma(t)}}$ never vanishes,
and $\gamma$ is properly embedded in $M_1$,
the vectors 
$\Vec{v_{\gamma(0)}}$ and $\Vec{v_{\gamma(1)}}$
point in opposite sides of $F$;
i.e. if $\Vec{v_{\gamma(0)}}$ points inside $M_1$ then 
$\Vec{v_{\gamma(1)}}$ points outside $M_1$ and vice-versa.
\\
Therefore (see Figure~1)
either~:

$\Vec{v_{\gamma(0)}}$ and $\Vec{n_{\gamma(0)}}$
both point in the same side of $M_1$, 
and,
$\Vec{v_{\gamma(1)}}$ and $\Vec{n_{\gamma(1)}}$
point opposite sides of $M_1$;
\\
or

$\Vec{v_{\gamma(1)}}$ and $\Vec{n_{\gamma(1)}}$
both point in the same side of $M_1$, 
and, 
$\Vec{v_{\gamma(0)}}$ and $\Vec{n_{\gamma(0)}}$
point in opposite sides of $M_1$.
\\

The former case means $h(0)<0$ and $h(1)>0$,
while the later case means the inverse.
Therefore $h(0)h(1)<0$.

\qed(Proposition~\ref{sep_torus})
\begin{THM}
[collection of all \cite{B2, EHN, Le, Ma, No, Th}]
\label{cpt_iso}
Let $M$ be a \QHS\ Seifert fibered $3$-manifold, 
with $n$ exeptional exceptional fibers (where $n\geq3$).
We assume that $M$ admits a taut \C0-foliation $\mathcal{F}$.
Moreover, if $n>3$, we suppose that $\mathcal{F}$ is a \c2-foliation of $M$.

If $\mathcal{F}$ does not have a compact leaf,
then $\mathcal{F}$ can be isotoped to be a horizontal foliation.
\end{THM}
\begin{RK}[History on Theorem~\ref{cpt_iso}]\label{history}
This theorem has been proved for all Seifert $3$-manifolds which are not
trivial bundles over the $2$-torus. This is a collection of results as follows.
\\
The case of circle bundles over orientable surface, which is not a $2$-torus is due to W.~P.~Thurston in \cite{Th};
it has been completed and extended to non-orientable base surface by G.~Levitt in \cite{Le}. 
\\
In \cite{EHN}, D. Eisenbud, U. Hirsch, and W. Neumann generalized it to Seifert fibered spaces, where the base surface is neither $\mathbb{S}^2$, nor  the $2$-torus with trivial circle bundle.
\\
Later, in \cite{Ma}, S. Matsumoto focused on the case when the base is $\mathbb{S}^2$ with stricly more than $3$ exceptional fibers.
\\
Until there,  the condition of $\mathcal{C}^r$-foliation is necessary, and implies a $\mathcal{C}^r$-isotopy, for each $r\geq 2$.
\\
The last case (the base is $\mathbb{S}^2$ with $3$ exceptional fibers) was solved by M.~Brittenham~\cite{B2}, and the involved techniques are very different, so the author obtained a $\mathcal{C}^0$-isotopy from a $\mathcal{C}^0$-foliation.
\\
We may recall that when there are one or two exceptional fibers with base $\mathbb{S}^2$, there is no foliation without compact leaf by S.~P.~Novikov in \cite{No}.
\end{RK}
\noindent
{\sc Alternative proof of Corollary~\ref{equiv cor}}
\\
A proof of Corollary~\ref{equiv cor} has been obtained by combining the results of 
Y.~Eliashberg and W.~P.~Thurston \cite{ET}, M.~Jankins and W.~Neumann \cite{JN},
P.~Lisca and G.~Mati\'c \cite{LM}, P.~Lisca and A.~I.~Stipsicz \cite{LS},
R.~Naimi \cite{N},
and,
P.~Ozsv\'ath and Z.~Szab\'o \cite{OS}, in the following way.
\begin{THM}[\cite{ET, JN, LM, LS, N, OS}]
Let $M$ be a rational Seifert fibered homology $3$-sphere.
The following statements are equivalent~:
\begin{itemize}
\item[(1)] $M$ is an $L$-space;
\item[(2)] $M$ does not carry a transverse contact structure;
\item[(3)] $M$ does not admit a transverse foliation;
\item[(4)] $M$ does not admit a taut foliation.
\end{itemize}
\end{THM}
This theorem is a formulation of \cite[Theorem~1.1]{LS}.
The proof is mainly organized as follows.
\begin{enumerate}
\item[$(1)\Rightarrow (2)$]~: P. Ozsv\'ath and Z. Szab\'o \cite{OS};
\item[$(2) \Rightarrow (3)$]~: Y. Eliashberg and W. P. Thurston \cite{ET};
\item[$(1) \Rightarrow (4)$]~: P. Ozsv\'ath and Z. Szab\'o \cite{OS};
\item[$(4) \Rightarrow (3)$]~: trivial;
\item[$(3) \Rightarrow(2)$]~: M. Jankins and W. Neumann \cite{JN}, P. Lisca and G. Mati\'c \cite{LM},
and R. Naimi \cite{N};
\item[$(2) \Rightarrow(1)$]~: P. Lipsca and A. I. Lipsicz \cite{LS}.
\end{enumerate}
To have $(3) \Rightarrow(4)$, we need to follow~:
$(3) \Rightarrow(2) \Rightarrow (1) \Rightarrow (4)$.
\\

We may underline that the considered taut foliations are actually \c2-foliations, because using contact structure
(see \cite{ET}). Note that there exists a taut \C0-foliation which is not a taut \c2-foliation \cite{BNR}.
\section{Characterization of taut \c2-foliations in Seifert fibered homology $3$-spheres}
The goal of this section is to give a characterization of the existence 
of a taut \c2-foliation in a $\mathbb{Q}$HS Seifert fibered $3$-manifold.
For this, we define the following {\it Property~$(*)$}.
$$
{\it Property}\ (*) \left\{ \begin{array}{ll}
(i)&\dps \frac{b_1}{a_1}<\frac{m-\alpha}{m}\\ 
\\
(ii)&\dps \frac{b_2}{a_2}<\frac{\alpha}{m}\\
\\
(iii)&\dps \frac{b_i}{a_i}<\frac{1}{m}  \textrm{,~for $i\in\{3,\dots,n\}$}
\end{array} \right.
$$
We say that {\it $m$ and $\alpha$ satisfy Property $(*)$ for
$b_1/a_1,b_2/a_2,\dots, b_n/a_n$},
if all the following statements are satisfied~:
\begin{itemize}
\item[-] $m$ and $\alpha$ are two positive integers such that $\alpha < m$;
\item[-] $n\geq 3$ is an integer;
\item[-] $a_i$ and $b_j$ are positive integers for all $(i,j)\in\{1,\dots,n\}^2$,
such that~:

\ \hfil $b_1/a_1\geq b_2/a_2\geq\dots\geq b_n/a_n$;\hfill\ \qquad
\item[-] $(i), (ii)$ and $(iii)$ of Property $(*)$ all are satisfied.
\end{itemize}
When there is no confusion for the $b_i/a_i$'s,
we say for short that {\it $(m,\alpha)$ satisfies Property $(*)$},
or that {\it the integers $\alpha$ and $m$ satisfy Property $(*)$}.
 
For convenience, in the following, 
we denote by $(i), (ii)$ and $(iii)$
respectively, the inequalities $(i), (ii)$ and $(iii)$ of Property $(*)$ above.
\\

Let $M$ be a Seifert fibered $3$-manifold.
In the following, we use the previous notations (see Section 2)
of Seifert normalized invariant~:
$$M=M(-b_0, b_1/a_1, \dots, b_n/a_n)$$
where
$a_i$ and $b_j$ are positive integers for all $(i,j)\in\{1,\dots,n\}\times\{0,\dots,n\}$,
such that $0<b_i<a_i$.
Note that the notations $M=M(-b_0, b_1/a_1, \dots, b_n/a_n)$ suppose that $M$ contains exactly $n$ exceptional fibers~: 
\begin{center}
$a_i\geq 2$, for all $i\in\{1,2,\dots,n\}$.
\end{center}

If $b_0\not\in\{1,n-1\}$ then the existence of a taut \c2-foliation depends uniquely of $b_0$,
as suggests the following theorem. 
\begin{THM}[\cite{EHN, JN, N}]\label{ehn}
Let $n$ be an integer and $M$ be a Seifert manifold based on $\mathbb{S}^2$.
\\
We assume that $n\geq 3$ and that $M=M(-b_0, b_1/a_1, \dots, b_n/a_n)$, where
$a_i$ and $b_j$ are positive integers for all $(i,j)\in\{1,\dots,n\}\times\{0,\dots,n\}$.
Then, all the following statements are satisfied.
\begin{enumerate}
\item
If $2\leq b_0\leq n-2$ then $M$ admits a horizontal foliation.
\item
If $M$ admits a horizontal foliation then $1\leq b_0\leq n-1$.
\item
If $M$ admits a horizontal \C0-foliation, then $M$ admits a horizontal analytic foliation.
\end{enumerate}
\end{THM}
 
\begin{COR}\label{hor}
Let $n$ be an integer and $M$ be a Seifert manifold based on $\mathbb{S}^2$.
\\
We assume that $n\geq 3$ and that $M=M(-b_0, b_1/a_1, \dots, b_n/a_n)$, where
$a_i$ and $b_j$ are positive integers for all $(i,j)\in\{1,\dots,n\}\times\{0,\dots,n\}$,
and $b_0\not\in\{1,n-1\}$.
\\
Then
$M$ admits an analytic horizontal foliation if and only if $2\leq b_0\leq n-2$.
\end{COR}
Therefore, the problem falls on $b_0=1$; we may recall here (see Section~$2$)~: 
$$M(-1, b_1/a_1, \dots, b_n/a_n)\cong-M(-(n-1), 1-b_1/a_1, \dots, 1-b_n/a_n).$$

The following theorem is a consequence of Corollary \ref{equiv cor}
and
the characterization of the existence of horizontal foliations in Seifert-fibered spaces based on 
$\mathbb{S}^2$, whose formulation can be found in \cite[Proposition 6]{BNR}.
\begin{THM}\label{exist_taut}
Let $n> 2$ be an integer and $M=M(-1, b_1/a_1, \dots, b_n/a_n)$
be a \QHS\ Seifert fibered $3$-manifold;
where $a_i$ and $b_j$ are positive integers for all $(i,j)\in\{1,\dots,n\}^2$.

Assume that $b_1/a_1\geq b_2/a_2\geq\dots\geq b_n/a_n$.

If $n> 3$ (resp. $n=3$)
then, $M$ admits a taut $\mathcal{C}^2$-foliation 
(resp. a taut \C0-foliation)
if and only if there exist two positive integers
$m$ and $\alpha$ such that $(m,\alpha)$ satisfies Property $(*)$.
\end{THM}

We may recall that $\mathcal{P}$ denotes the Poincar\'e $\mathbb{Z}$HS, i.e. $\mathcal{P}=M(-1,1/2,1/3,1/5)$.
Note that Theorem \ref{exist_taut} implies that $\mathcal{P}$ cannot admit a taut foliation,
but this fact was already known by S.~P.~Novikov's Theorem, see \cite{No} (because its $\pi_1$ is finite).
Note also that if $n\in\{1,2\}$ then $M$ has to be $\mathbb{S}^3$
or a Lens space, which cannot admit a taut foliation.
\\ 
Theorem \ref{exist_taut} has the following corollaries, which will be useful for the next sections.
 \begin{COR}\label{1/2}
 Let $n$ be an integer and $M$ be a \QHS\ Seifert fibered $3$-manifold.\\
We assume that $n\geq 3$ and that $M=M(-1, b_1/a_1, \dots, b_n/a_n)$, where
$a_i$ and $b_j$ are positive integers for all $(i,j)\in\{1, \dots, n\}^2$.

We order the rational coefficients $b_i/a_i$ such that~:
$b_1/a_1\geq b_2/a_2\geq \dots \geq b_n/a_n$.	
\\
If for all $i\in\{1,\dots,n\}$, $\dps \frac{b_i}{a_i}<\frac1 2$ then $M$ admits a taut \c2-foliation.
\end{COR}
\begin{proof}
With the notations and assumptions of the theorem,
if $b_i/a_i<1/2$, for all $i\in\1n$, then Property $(*)$ is satisfied, by choosing $m=2$ and $\alpha=1$.
\end{proof}
 \begin{COR}\label{1/3}
 Let $n$ be an integer and $M$ be a \QHS\ Seifert fibered $3$-manifold.
\\
We assume that $n\geq 3$ and that $M=M(-1, b_1/a_1, \dots, b_n/a_n)$, where
$a_i$ and $b_j$ are positive integers for all $(i,j)\in\{1,\dots,n\}^2$.

We order the rational coefficients $b_i/a_i$ such that~:
$b_1/a_1\geq b_2/a_2\geq \dots\geq b_n/a_n$.

If $M$ admits a taut \c2-foliation and $\dps \frac{b_1}{a_1}\geq1/2$,
then the two following properties are both satisfied.
\\
\begin{enumerate}
\item 
$\dps \frac{b_i}{a_i}<\frac 1 2$, for all $i\geq 2$..
\\
\item
$\dps\frac{b_n}{a_n}<\frac13$.
In particular, $a_n\geq 4$.
\end{enumerate}
\end{COR}
\begin{proof}
With the notations and assumptions of Theorem~\ref{exist_taut}, 
if $M$ admits a taut \c2-foliation then we can find
positive integers $m,\alpha$ such that $\alpha<m$ and Property~$(*)$ 
is satisfied.

First, note that if $m=2$ then $\alpha=1$ and $b_1/a_1<1/2$, which is a contradiction to the 
hypothesis. Thus, $m\geq 3$.

Now, if $\dps \frac{m-\alpha}{m}> \frac{1}{2}$ then $\dps \frac{\alpha}{m}< \frac{1}{2}$, hence Property $(*)$ implies $\dps \frac{b_i}{a_i}<\frac{1}{2}$ for $i\in\{2,\dots,n\}$ which proves (1).

Finally, assume that 
$\dps \frac{b_1}{a_1}\geq \frac1 2\geq \frac{b_2}{a_2}\geq \frac{b_3}{a_3}\geq\frac 1 3$.
Then $b_3/a_3\geq1/m$ for all $m\geq 3$,
so $(iii)$ of $(*)$ cannot be satisfied.
 \end{proof}
\section{Geometries of Seifert fibered homology $3$-spheres}
The goal of this section is to recall general results on the 
geometries of Seifert fibered homology $3$-spheres,
and prove Proposition~\ref{global}.

Let $n$ be a positive integer and $M=M(-b_0, b_1/a_1, \dots, b_n/a_n)$ be a $\mathbb{Q}$HS Seifert 
fibered $3$-manifold.
Recall that $e(M)$ denotes the Euler number of $M$, see Section~\ref{prel}.
The following lemma is a well known result, see \cite{Gi} for more details.
\begin{LM}\label{QHS}
Let $M=M(-b_0, b_1/a_1, \dots, b_n/a_n)$ be a Seifert fibered $3$-manifold. Then~:
\begin{enumerate}
\item
$M$ is a $\mathbb{Z}$HS if and only if 
$\dps a_1a_2 \dots a_ne(M)= \varepsilon$, where $\varepsilon\in\{-1,+1\}$;

\item
$M$ is a $\mathbb{Q}$HS if and only if 
$\dps e(M)\not=0$.
\end{enumerate}
\end{LM}
\begin{RK}\label{rk ZHS}
Note that $(1)$ implies that the $a_i$'s are pairwise relatively prime integers,
therefore they are different. 
\end{RK}
Then, we define the rational number $\chi_M$ as follows.
$$\displaystyle \chi_M= 2 - \sum_{i=1}^n (1-\frac{1}{a_i})=2-n+\sum_{i=1}^n\frac{1}{a_i}.$$
We have the following well known result (which can be found in \cite{Sc} for example).
\begin{PR}\label{prop chi}
Let $n$ be a positive integer and $M=M(-b_0, b_1/a_1, \dots, b_n/a_n)$ be a $\mathbb{Q}$HS\
Seifert fibered $3$-manifold, then the following properties all are satisfied.
\begin{itemize}
\item[(i)]
$\chi_M>0 \Leftrightarrow$ $M$ admits the $\mathbb{S}^3$-geometry.
\item[(ii)]
$\chi_M<0 \Leftrightarrow$ $M$ admits the $\widetilde{SL}_2(\mathbb{R})$-geometry.
\item[(iii)]
$\chi_M=0 \Leftrightarrow$ $M$ admits the $\mathcal{N}$il-geometry.
\end{itemize}
\end{PR}
\begin{PR}\label{4}
Let $n$ be a positive integer and $M=M(-b_0, b_1/a_1, \dots, b_n/a_n)$ be a $\mathbb{Q}$HS\
Seifert fibered $3$-manifold.
If $M$ does not admit the $\widetilde{SL}_2(\mathbb{R})$-geometry then $n\leq 4$.

Furthermore, if $n=4$ then $M=M(-b_0,1/2,1/2,1/2,1/2)$ with $b_0\not=2$; 
so $M$ admits the $\mathcal{N}$il-geometry and is a non-integral \QHS.
\end{PR}
\begin{proof}
Let $n$ be a positive integer and $M=M(-b_0, b_1/a_1, \dots, b_n/a_n)$ be a \QHS.
Assume that $M$ does not admit the $\widetilde{SL}_2(\mathbb{R})$-geometry.
Then, by Proposition~\ref{prop chi}, $\chi_M\geq 0$.
Therefore, $\dps n-2\leq\sum_{i=1}^n\frac{1}{a_i}$.
\\
Since, $a_i\geq2$ for all $i\in\1n$,
$\dps n-2\leq\sum_{i=1}^n\frac{1}{a_i}\leq n/2\Rightarrow n\leq 4$.
\\\\
Now, assume first that $n=4$.
Then, $\dps \sum_{i=1}^4\frac{1}{a_i}\geq 2$.
On the other hand, $a_i\geq 2$ for all $i\in\{1,\dots,4\}$, then 
$\dps \sum_{i=1}^4\frac{1}{a_i}\leq 2$, 
and if one $a_i>2$ then $\dps \sum_{i=1}^4\frac{1}{a_i}< 2$.

Therefore, $a_i=2$ for all $i\in\{1,\dots,4\}$.
Thus, $\chi_M=0$ which means that $M$ admits the $\mathcal{N}$il-geometry.
Moreover, Lemma~\ref{QHS}~$(2)$ implies that $b_0\not=2$.
Note that such $M$ cannot be a \ZHS, by Remark~\ref{rk ZHS}.
\end{proof}
\begin{COR}\label{poincsphere}
Let $M$ be a \ZHS\ Seifert fibered $3$-manifold.
Then, $M$ has the $\widetilde{SL}_2(\mathbb{R})$-geometry 
or the $\mathbb{S}^3$-geometry.

Furthermore,
if $M$ has the $\mathbb{S}^3$-geometry, 
then $M$ is either homeomorphic to $\mathbb{S}^3$
or to the Poincar\'e sphere $\mathcal{P}$.
\end{COR}
\begin{proof}
Let $M$ be a \ZHS\ Seifert fibered $3$-manifold.
Assume that $M$ does not have the $\widetilde{SL}_2(\mathbb{R})$-geometry.
Note that if $n\leq 2$ then $M$ has to be homeomorphic to $\mathbb{S}^3$.
By Proposition~\ref{4}, we may assume that $n=3$ and that
$a_3>a_2>a_1\geq 2$ (by remark~\ref{rk ZHS}).
 
 Since $\chi_M\geq 0$, $\dps\sum_{i=1}^3\frac{1}{a_i}\geq1$.
 If $a_1\geq 3$, then $\dps\sum_{i=1}^3\frac{1}{a_i}\leq 1/3+1/4+1/5<1$, which is a contradiction.
 Then $a_1=2$.
 If $a_2\not=3$ then $a_2\geq 5$ by remark~\ref{rk ZHS}.
 Hence, $\dps\sum_{i=1}^3\frac{1}{a_i}\leq 1/2+1/5+1/7<1$, which is a contradiction.
 Therefore, $a_1=2$ and $a_2=3$. Similarly $a_3=5$.
 Since $n=3$ and $(a_1,a_2,a_3)=(2,3,5)$, $M$ has to be homeomorphic to the Poincar\'e sphere, which satisfies $\chi_M>0$, so $\mathcal P$ has the 
 $\mathbb{S}^3$-geometry.
 \end{proof}
To end this section, we simply note that Proposition~\ref{4} together with 
Corol\-lary~\ref{poincsphere} clearly imply Proposition~\ref{global}.
\section{Proof of Theorem \ref{geom_notaut}}
We keep the previous notations.
Let $n$ be a positive integer and $M$ be a $\mathbb{Q}$HS\
Seifert fibered $3$-manifold, with $n$ exceptional fibers~:
$M=M(-b_0, b_1/a_1, \dots, b_n/a_n)$.
Assume that $M$ does not admit the $\widetilde{SL}_2(\mathbb{R})$-geometry.
We make the proof by contradiction.
Suppose that $M$ admits a taut \c2-foliation.
We may recall that if $n\in\{1,2\}$,
then $M$ has a finite $\pi_1$, hence $M$ cannot admit a taut \c2-foliation.
Therefore, by Proposition~\ref{4}, we have $n\in\{3,4\}$.
\\
Assume that $n=4$.
By Theorem~\ref{ehn} and Theorem \ref{eq_taut_hor}, since $M$ admits a taut \c2-foliation,
$b_0\in \{1,2,3\}$.
Moreover the cases $b_0=1$ and $b_0=3$ are equivalent
(see the fiber-preserving homeomorphism~$\Phi$ in Section~$2$).

On the other hand, Proposition~\ref{4} implies that $M=M(-b_0,1/2,1/2,1/2,1/2)$ with $b_0\not=2$;
and Corollary~\ref{1/3} $(1)$ implies that $b_0\not=1$.

Therefore, we may assume that $n=3$.
Similarly $b_0\in\{1,2\}$ and $b_0=1$ and $b_0=2$ are equivalent cases, 
by considering the fiber-preserving homeomorphism~$\Phi$.

So, we may assume that $b_0=1$.
Let $M=M(-1,b_1/a_1,b_2/a_2,b_3/a_3)$,

Since $M$ is a $\mathbb{Q}$HS Seifert fibered $3$-manifold,
which does not admit the $\widetilde{SL}_2(\mathbb{R})$-geometry,
Proposition~\ref{prop chi} and Lemma~\ref{QHS}(2) give respectively~:
 $$
\left\{ \begin{array}{ll}
\hbox{\rm (I1)}&\dps \sum_{i=1}^3\frac{1}{a_i}\geq 1\\
\\
\hbox{\rm and}
\\
\hbox{\rm (I2)}&\dps\sum_{i=1}^3\frac{b_i}{a_i}\neq 1\\
\end{array} \right.
$$
\\
By Corollary~\ref{1/3}, we order the coefficients~:
$b_1/a_1\geq b_2/a_2\geq b_3/a_3$.

Let $a_{i_0}=\min(a_1,a_2,a_3)$. By (I1), $a_{i_0}\in\{2,3\}$.

First, we prove that $a_{i_0}$ cannot be $3$.
We make the proof by contradiction.
Assume that $a_{i_0}=3$, then (I1) implies that $a_i=3$ for all $i\in\{1,2,3\}$.
Now, for all $i\in\{1,2,3\}$, $b_i<a_i$  so  $b_i\leq 2$.
If there exists $i\in\{1,2,3\}$ such that $b_i=2$ then $b_i/a_i=2/3>1/2$.
But for $j\neq i$, $b_j/a_j\geq 1/3$, which is a contradiction
to Corollary~\ref{1/3} (2).
Therefore, $b_i/a_i=1/3$, for all $i\in\{1,2,3\}$,
which contradicts~(I2).

Hence, we may assume that $a_{i_0}=2$.
Then~:
$$b_{i_0}/a_{i_0}=1/2.$$
By Corollary~\ref{1/3} $(1)$, $b_{i_0}/a_{i_0}=b_1/a_1$.

Then Corollary \ref{1/3} $(1)$ and $(2)$ imply respectively that
$a_3\geq 4$
and $a_2\geq 3$.

Now (I1) implies that $\{a_1,a_2,a_3\}$ is one of the following sets~:
\begin{center}
$\{2,3,4\}, \{2,3,5\},\{2,3,6\}$  or $\{2,4,4\}$.
\end{center}

We distinguish the cases $a_{2}=3$ and $a_{2}=4$.
\\\\
{\sc Case 1~:} $a_{2}=3$.
\\
Then Corollary \ref{1/3} (1) implies that~:
$$b_{2}/a_{2}=1/3.$$
Now, by Theorem~\ref{exist_taut}, there exist positive integers 
$\alpha$ and $m$ which satisfy Pro\-perty~$(*)$.
Now Corollary \ref{1/3} (2) implies that~:
$\dps \frac{b_{3}}{a_{3}}\in\bigg\{\frac14, \frac15, \frac16\bigg\}$.
Hence, by $(*) (iii)$, $m\leq 5$. 

Since $b_{1}/a_{1}=1/2$, $m>2$.

If $m=3$ then $\alpha\in\{1,2\}$, but in both cases $(*) (i)$ or $(*) (ii)$ cannot be satisfied.
Similarly, if $m=4$ then $\alpha\in\{1,2,3\}$,
but in all cases $(*) (i)$ or $(*) (ii)$ cannot be satisfied.

If $m=5$ then $a_{3}=6$ and $b_{3}=1$; otherwise $(*) (iii)$ cannot be satisfied.

Thus, $b_1/a_1+b_2/a_2+b_3/a_3=1/2+1/3+1/6=1$, which is in contradiction to~$(I2)$,
i.e. $M$ cannot be a \QHS.
\\\\
{\sc Case 2~:} $a_{2}=4$.
\\
Then $a_{2}=a_{3}=4$. Therefore Corollary \ref{1/3}~$(1)$ implies that
$\dps \frac{b_{2}}{a_{2}}= \frac{b_{3}}{a_{3}}=\frac14$.
Therefore $(I2)$ is not satisfied, which is the final contradiction.

This ends the proof of Theorem~\ref{geom_notaut}.
%
\section{Proof of Main Theorem~$2$}
Let $n$ be a positive integer greater than two.
We keep the previous conventions and notations
and denote any \QHS\ Seifert fibered $3$-manifolds $M$ with its normalized Seifert invariant,
by~:
$M=M(-b_0, b_1/a_1,b_2/a_2, \dots,b_n/a_n)$.

Let $\mathcal{SF}_1$ be the set of all Seifert fibered $3$-manifolds
for which $b_0=1$
and which admit the $\widetilde{SL}_2(\mathbb{R})$-geometry.
\\
\\
We denote by $\mathcal{Q}_n$ the set~:
$$\mathcal{Q}_n=\mathcal{S}_n\cap\mathcal{SF}_1.$$
Then $\mathcal{Q}_n$ is the set of non-integral
\QHS\ Seifert fibered $3$-manifolds $M$ with 
$n$ exceptional fibers, which admit the $\widetilde{SL}_2(\mathbb{R})$-geometry
and 
$M=M(-1, b_1/a_1,b_2/a_2, \dots,b_n/a_n)$.

This section is devoted to prove the following result, which clearly implies Main Theorem~$2$.
\begin{THM}\label{geom_taut rational}
Let $n$ be a positive integer greater than two.
For each $n$~:
\begin{itemize}
\item[(i)]
There exist infinitely many Seifert fibered manifolds in $\mathcal{Q}_n$ which admit a taut analytic foliation; and
\item[(ii)]
There exist infinitely many Seifert fibered manifolds in $\mathcal{Q}_n$ which do not admit a taut \c2-foliation.
\item[(iii)]
There exist infinitely many Seifert fibered manifolds in $\mathcal{Q}_3$ which do not admit a taut \C0-foliation.
\end{itemize}
\end{THM}
\begin{proof}
The proof of Theorem \ref{geom_taut rational} is an immediate consequence of the two following lemmata.
Let $n$ be a positive integer greater than two.
Let  $\mathcal{M}(n)$ be the family of 
Seifert fibered $3$-manifolds $M$ with $n$ exceptional fibers such that
$\dps M= M(-1,\frac1 2,\frac{b_2}{a_2},\frac{b_3}{a_3}, \dots,\frac{b_n}{a_n})$
and 
the exceptional slopes are ordered in the following way~:
$\dps\frac12>\frac{b_2}{a_2}\geq\frac{b_3}{a_3}\geq\dots\geq\frac{b_n}{a_n}$.
\begin{LM}\label{infinite fam1}
Let $n$ be a positive integer greater than two.
We consider the following families of infinite 
Seifert fibered $3$-manifolds.

\begin{center}
$\dps\mathcal{M}_1(n)=\bigg\{M\in\mathcal{M}(n),\ \textit{with\ }
\frac{b_2}{a_2}=\frac35, \ n> 3\bigg\}
$;

$\dps\mathcal{M}_1(3)=\bigg\{M\in\mathcal{M}(3),\ \textit{with\ }\frac{b_2}{a_2}=\frac35, \hbox{\ and\ } a_3\geq 4\bigg\}
$;

$\dps\mathcal{M}_2(n)= \bigg\{M\in\mathcal{M}(n),\ \textit{with\ }
\frac{b_2}{a_2}=\frac25,\ \dps \frac{b_3}{a_3}> \frac1{5}, n> 3\bigg\}
$;

$\dps\mathcal{M}_2(3)= \bigg\{M\in\mathcal{M}(3),
\ \textit{with\ }\frac{b_2}{a_2}=\frac25,
\ \dps \frac{b_3}{a_3}> \frac1{5},
\hbox{\ and\ }a_3\geq 4\bigg\}
$.
\end{center}
If $M\in\mathcal{M}_1(n)\cup\mathcal{M}_2(n)$,
then $M\in\mathcal{Q}_n$.
In particular, $M$  is a non-integral homology $3$-sphere,
which admits the $\widetilde{SL}_2(\mathbb{R})$-geometry,
and $M$ does not admit a taut \c2-foliation.
\\\\
Furthermore, if $M\in\mathcal{M}_1(3)\cup\mathcal{M}_2(3)$,
then $M\in\mathcal{Q}_3$,
and $M$ does not admit a taut \C0-folitation.
\end{LM}
\begin{proof}
First, considering Lemma~\ref{QHS},
we may check easily that if $M\in\mathcal{M}_1(n)\cup\mathcal{M}_2(n)$ then $M$ is a \QHS\ but not a \ZHS.

Indeed if $M\in\mathcal{M}_1(n)$, then
 $\dps e(M)>-1+1/2+3/5$.

If $M\in\mathcal{M}_2$(n), then
 $\dps e(M)>-1+1/2+2/5+1/5$.

In both cases, $e(M)>1/10$, so $e(M)\not=0$; hence, $M$ is a \QHS.
 
On the other hand, if $\dps e(M)=\frac{\varepsilon}{a_1a_2\dots a_n}$ 
(where $\varepsilon=\pm 1$)
then $\dps e(M)<\frac1{10a_3}$; 
\\

which is a contradiction.
Then, $M$ is not a \ZHS.
\\\\
Now, we check that they all have the $\widetilde{SL}_2(\mathbb{R})$-geometry.

If $n\geq 4$, then it is a direct consequence of Proposition ~\ref{4}.

If $n=3$, that follows from $\dps \sum_{i=1}^3\frac{1}{a_i}<1$ (here, we need that $a_3\geq 4$).

In conclusion, $\mathcal{M}_1(n)\cup\mathcal{M}_2(n)
\subset \mathcal{Q}_n$ (for $n\geq 3$).
\\\\
Finally, we check that they do not admit a taut \c2-foliation.

If $M\in\mathcal{M}_1(n)$, 
Corollary \ref{1/3} $(1)$ implies that $M$ cannot admit a taut \c2-foliation.

If $M\in\mathcal{M}_2(n)$, then $\dps\frac{b_2}{a_2}$ and $\dps\frac{b_3}{a_3}$
both are greater than $1/5$; therefore $(iii)$ implies that $m\leq 4$.
Thus, $\alpha\in\{1, 2, 3\}$. In all cases, $(i)$ or $(ii)$ cannot be satisfied.
\\\\
Furthermore, by Theorem~\ref{eq_taut_hor},
if $M\in\big(\mathcal{M}_1(3)\cup\mathcal{M}_2(3)\big)$ and 
$M$ admits a taut \C0-folitation, then the foliation can be isotoped to be horizontal;
which is impossible for $M$ in $\mathcal{M}_1(3)\cup\mathcal{M}_2(3)$.
\end{proof}
\begin{LM}\label{infinite fam2}
Let $n$ be a positive integer greater than two.
Let  $\mathcal{M}_3$ and $\mathcal{M}_4(n)$ be the two following families of infinite 
Seifert fibered $3$-manifolds.
\begin{center}
$\dps\mathcal{M}_3=\bigg\{M\big(-1, \frac1 2, \frac2 5, \frac{k}{7k+1} \big)
\in\mathcal{M}(3),\ k\in\mathbb{Z},\ k\geq 1\bigg\}$;

$\dps\mathcal{M}_4(n)= \bigg\{
M\big(-1, \frac1 2, \frac2 5,\frac{1}{10}, \frac{b_4}{10b_4+1},
\dots, \frac{b_n}{10b_n+1} \big)\in\mathcal{M}(n)
,\ n>3\bigg\}$.
\end{center}
If $M\in\mathcal{M}_3\cup\mathcal{M}_4(n)$,
then $M\in\mathcal{Q}_n$ and is a non-integral Seifert fibered $3$-manifold,
which admits the $\widetilde{SL}_2(\mathbb{R})$-geometry
and a taut analytic foliation.
\end{LM}
\begin{proof}
First, considering Lemma~\ref{QHS},
we can check that if $M\in\mathcal{M}_3\cup\mathcal{M}_4$,
then $M$ is a \QHS\ but not a \ZHS.

Indeed,
if $M\in\mathcal{M}_3$, then
 $\dps e(M)>-1+1/2+2/5+1/8$, 
 
 i.e. $e(M)>1/40$; so $e(M)\not=0$ and $M$ is a \QHS.
 
If $\dps e(M)=\frac{\varepsilon}{a_1a_2a_3}$ 
(where $\varepsilon =\pm 1$)
then $e(M)<1/70$;
\\

which is not possible so $M$ is not a \ZHS.
\\

 Similarly,
 if $M\in\mathcal{M}_4$, then
 $\dps e(M)>-1+1/2+2/5+1/10+1/11$, 
 
 i.e. $e(M)>1/11$; so $e(M)\not=0$ and $M$ is a \QHS.

If $\dps e(M)=\frac{\varepsilon}{a_1a_2\dots a_n}$ then $e(M)<1/100$,
which is  not possible so $M$ is not a \ZHS.
\\\\
Now, we check that they all admit the $\widetilde{SL}_2(\mathbb{R})$-geometry.

If $n\geq 4$, then it is a direct consequence of Proposition ~\ref{4}.

If $n=3$, that follows from $\dps \sum_{i=1}^n\frac{1}{a_i}<1$.
\\\\
Finally, if we choose $\alpha=3$ and $m=7$ 
then $(m, \alpha)$ trivially satisfies Property~$(*)$;
which implies that they all admit a taut analytic foliation
(by Theorems~\ref{ehn} and \ref{exist_taut}).
\end{proof}
\ \hfill{\it End of proof of Theorem \ref{geom_taut rational}}

\end{proof}
%
\section{Proof of Main Theorem $1$}
This section is almost entirely devoted to the proof of 
Proposition~\ref{main}, which implies Main Theorem $1$,
as it will be shown bellow.

We may recall here (see Section~$2$)
that if $M$ is a Seifert fibered $3$-manifold,
then $M=M(-b_0, b_1/a_1, \dots, b_n/a_n)$, where $b_0$ is a positive integer and
$0<b_i<a_i$ for all $i\in\{1,\dots,n\}$.
Note that $n$ has to be greater than $2$ (otherwise $M$ cannot be a \ZHS\
but $\mathbb{S}^3$).

If $M$ is also a \ZHS, then two rational coefficients cannot be the same,
see Remark~\ref{rk ZHS};
therefore we may re-order them so that
$b_1/a_1> b_2/a_2>\dots> b_n/a_n$.

Thus, two positive integers $m$ and $\alpha$ satisfy Property~$(*)$
(for these rational coefficients) if and only if~:
$$
\fbox{
\vbox{\hsize8.3truecm
\noindent
$\alpha < m$ and 
$(i)$ to $(iii)$ of Property~$(*)$ are all satisfied.
}}
$$
\begin{PR}\label{main}
Let $n$ be a positive integer and $M$ be a \ZHS\ Seifert fibered $3$-manifold,
which is neither homeomorphic to $\mathbb{S}^3$ nor to $\mathcal{P}$.
\\
We assume that $M=M(-1, b_1/a_1, \dots, b_n/a_n)$,
where~:

- $0<b_i<a_i$ for all $i\in\{1,\dots,n\}$,
and;

- $b_1/a_1> b_2/a_2>\dots> b_n/a_n$.
\\
Then there exist two positive integers
$m$ and $\alpha$ which satisfy Property~$(*)$.
\end{PR}
\bigskip
\begin{center}
{\sc Proof of Main Theorem 1}
\end{center}
First of all, if $M$ is either homeomorphic to $\mathbb{S}^3$ or to the Poincar\'e sphere 
$\mathcal{P}$, then we may recall that $M$ cannot admit a taut foliation.

We assume that $M$ is neither homeomorphic to $\mathbb{S}^3$ 
nor to the Poincar\'e sphere 
$\mathcal{P}$. 
We want to show that $M$ always admit a taut analytic foliation.

Let $M=M(-b_0, b_1/a_1, \dots, b_n/a_n)$,
where $b_0$ is a positive integer and
$0<b_i<a_i$ for all $i\in\{1,\dots,n\}$.

First, we may note that
Corollary~\ref{hor} claims that if 
$b_0\in\{2,\dots,n-2\}$ then $M$ admit a horizontal analytic foliation,
which are taut \c2-foliation.
Then, we assume for the following that $b_0\not\in\{2,\dots,n-2\}$.
 
On the other hand, since $M$ is a \ZHS,
Lemma~\ref{QHS} $(1)$ implies that
$$b_0=\dps\sum_{i=1}^{n}\frac{b_i}{a_i}+\frac{\varepsilon}{a_1a_2\dots a_n},
\hbox{\ where $\varepsilon\in\{-1,+1\}$}.$$
Then, the property $0<b_i/a_i< 1$ for all $i\in\1n$,
implies that $0<b_0<n$.
By the fiber-preserving homeomorphism $\Phi$ (see Section~$2$)
we may assume that $b_0=1$.
Hence, Proposition~\ref{main} implies that there exists a pair of positive integers
$(m,\alpha)$ which satisfy Property~$(*)$.
This implies that $M$ admits a horizontal foliation (Theorem~\ref{exist_taut})
then a taut analytic foliation (Lemma~\ref{easy} and Theorem~\ref{ehn});
which ends the proof of Main Theorem~1.
\qed
\\\\
The remaining of the paper is entirely devoted to the
proof of Proposition~\ref{main}.
\\\\
{\bf Schedule of the proof of Proposition~\ref{main}}
\\
The proof of Proposition~\ref{main} is organized in four steps, as follows.
\begin{itemize}
\item[\bf Step 1~:]
If Proposition \ref{main} is true for $n=3$ then it is true for all $n\geq 3$.
\item[\bf Step 2~:]
Considering $n=3$ gives common notations
and results for the following.
\item[\bf Step 3~:]
We show Proposition \ref{main} for $n=3$ and $\epsilon =-1$.
\item[\bf Step 4~:]
We show Proposition \ref{main} for $n=3$ and $\epsilon =1$.
\end{itemize}

\medskip
Before starting the proof, we fix some notations
and conventions for all the following of the paper.
\\\\
{\bf Notations - Conventions}
\\
We keep the previous notations.

Let $M=M(-1, b_1/a_1, \dots, b_n/a_n)$ be a \ZHS\
Seifert fibered $3$-manifold,
where
$0<b_i<a_i$ for all $i\in\{1,\dots,n\}$.
\\

By Lemma \ref{QHS}, $M$ is a \ZHS\ if and only if~:
\\

\begin{center}
$(E1)\qquad \dps \sum_{i=1}^n \frac{b_i}{a_i}= 1+
\frac{\epsilon}{a_1.a_2.\dots.a_n}$, \ where $\epsilon\in\{-1,1\}$
\end{center}

\medskip\noindent
Let $\hat{a}_i$ (for $i\in\1n$), $\alpha_1, \alpha_2, a'_3, b'_3$ be the following positive rational numbers.
Note that all are positive integers but $\alpha_1, \alpha_2$, which are rational numbers.
$$
\fbox{
\vbox{\hsize5.1truecm
\noindent
$
\begin{array}{cc}
\dps \alpha_1=1-\frac{b_1}{a_1}&
\dps \alpha_2=\frac{b_2}{a_2}\\
\\
\dps\hat{a}_i= \frac{a_3\dots a_n}{a_i} &\forall i\in\{3,\dots,n\}\\
\\
\dps b_3'=\sum_{i=3}^nb_i\hat{a}_i&
\dps a_3'=a_3 \dots a_n\\
\end{array}
$
}}
$$
Thus, 
$$\dps \frac{b_3'}{a_3'}=\sum_{i=3}^n \frac{b_i}{a_i}.$$

\medskip
\noindent
Now, we fix the following inequalities by denoting them from $(1)$ to $(6)$.
The former three are trivially always true.
The last three are true when $n=3$,
see Claim~\ref{n=3}; they concern Steps $2$ to~$4$.

$$
\fbox{
\vbox{\hsize5.7truecm
\noindent
$
\begin{array}{c}
(1) \quad\dps\frac{b_1}{a_1}>\frac{b_2}{a_2}>\dots>\frac{b_n}{a_n}\\
\\
(2) \quad\dps\frac{b_1}{a_1}\geq\frac1 2\qquad \qquad
(3) \quad\dps\alpha_1\leq \frac{b_1}{a_1}\\
\end{array}
$
}}
$$

$$
\fbox{
\vbox{\hsize10.1truecm
\noindent
When $n=3$~:
\\\\
\noindent
$
\begin{array}{c}
(4) \quad \dps \frac{b_2}{a_2}<\frac1 2 \qquad \qquad
(5) \quad\dps \frac{b_3}{a_3}<\frac1 4
\qquad \qquad
(6)\quad \alpha_2>\alpha_1-\alpha_2\\
\end{array}
$
}}
$$

\noindent
$(1)$ up to reordering;
\\
$(2)$ by Corollary~\ref{1/2}, which implies $(3)$;
\\
$(4)$ to $(6)$, by Claim~\ref{n=3}.

\medskip
\noindent
When $n=3, \ a'_3=a_3$ and $b'_3=b_3$, and $(E_1)$ is equivalent to~:

$$
\fbox{
\vbox{\hsize8.6truecm
$
\noindent
(E2)\qquad \dps\frac{b_3}{a_3}=\alpha_1-\alpha_2+ \frac{\epsilon}{a_1a_2a_3},
\hbox{\it \ where\ } \epsilon\in\{-1,1\}.
$
}}
$$

\begin{CL}\label{n=3}
If $n=3$ then 
$\dps \frac{b_2}{a_2}<\frac1 2$,
$\dps \frac{b_3}{a_3}<\frac1 4$ and $\alpha_2>\alpha_1-\alpha_2$.
\end{CL}
\begin{proof}
Since $b_1/a_1\geq1/2$, there exists a non-negative integer $r_1$ such that
$$2b_1= a_1+r_1.$$
If $b_2/a_2+b_3/a_3<1/2$ then $(1)$ implies $(4)$ and $(5)$.
So, we may suppose that $b_2/a_2+b_3/a_3\geq1/2$.
Hence, there exists a non-negative integer $r$ such that~:
$$2(b_2a_3+a_2b_3)=a_2a_3+r.$$
Then $(E2)$ implies that
$\displaystyle\frac{r_1}{2a_1}+\frac{r}{2a_2a_3}= \frac{\epsilon}{a_1a_2a_3}$,  so
$r_1a_2a_3+ra_1=2\epsilon$.
\\\\
Therefore, $r_1= 0,\ r=1,\ a_1=2$, and $\epsilon=+1$. 
\\
Thus, $b_1/a_1=1/2$ and $(1)$ implies $(4)$ and $a_3b_2>a_2b_3$.
\\
Then $2(b_2a_3+a_2b_3)=a_2a_3+r$ implies
$1+a_2a_3>4a_2b_3$ and so $a_2a_3\geq 4a_2b_3$, which is equivalent to
$1/4\geq b_3/a_3$.
\\
Since $a_1=2$ and the $a_i$'s are pairwise relatively prime, 
$1/4> b_3/a_3$ which proves~$(5)$.
\\\\
By $(E2)$, $\alpha_1-\alpha_2=\dps\frac{b_3}{a_3}-\frac{\epsilon}{a_1a_2a_3}$.
\\\\
On the other hand, $(1)$ implies~: $b_2a_3\geq b_3a_2+1$ (since they are positive integers).
\\\\
Therefore,
$\dps\alpha_2=\frac{b_2}{a_2}\geq\frac{b_3}{a_3}+\frac{1}{a_2a_3}>\frac{b_3}{a_3}-\frac{\epsilon}{a_1a_2a_3}$ which implies $(6)$.
\end{proof}
\subsection{Step 1~: From $n=3$ to $n>3$}
\quad
\\
We suppose that Proposition \ref{main} is satisfied for $n=3$.
Now, we assume that $n\geq 4$ and $M=M(-1, b_1/a_1,\dots, b_n/a_n)$ is a \ZHS.
We want to show that Property $(*)$ is satisfied for the rational coefficients of 
the Seifert invariant of $M$.
\\
Let $M'=M(-1, b_1/a_1,b_2/a_2,b_3'/a_3')$.
Note that $(E_1)$ is satisfied because $M$ is a \ZHS;
therefore $M'$ is also a \ZHS, by the definition of $b_3'/a_3'$.
\\

We separate the proof according to either
$\displaystyle\frac{b_3'}{a_3'}<\frac{b_2}{a_2}$, or
$\displaystyle\frac{b_2}{a_2}<\frac{b_3'}{a_3'}$.

Note that $\dps \frac{b_2}{a_2}\neq\frac{b_3'}{a_3'}$
because the $a_i$'s are pairwise relatively prime.
\\\\
{\sc Case $1$~:}
$\displaystyle\frac{b_3'}{a_3'}<\frac{b_2}{a_2}$.
\\\\
First, we check that $M'\not\cong \mathcal{P}$.
Indeed, otherwise
$\dps \frac{b_3'}{a_3'}=\frac1 5$, so $a_3'=a_3\dots a_n=5$, with $n\geq 4$;
a contradiction.
Then, there exist positive integers $m$ and $\alpha$ such that  $\alpha<m$ and~:
\begin{itemize}
\item[(i)]
$\displaystyle\frac{b_1}{a_1}<\frac{m-\alpha}{m}$;\\
\item[(ii)]
$\displaystyle\frac{b_2}{a_2}< \frac{\alpha}{m}$; and\\
\item[(iii)]
$\displaystyle\frac{b_3'}{a_3'}<\frac{1}{m}$.
\end{itemize}
By definition, $\dps\frac{b_i}{a_i}<\frac{b_3'}{a_3'}$  for $i\in\{3,4,\dots,n\}$,
then the same positive integers $m$ and $\alpha$
satisfy Property $(*)$ for 
the rational coefficients $\dps \frac{b_i}{a_i}$ (for $i\in\1n$).
\\\\
{\sc Case $2$~:}
$\displaystyle\frac{b_2}{a_2}<\frac{b_3'}{a_3'}$.
\\\\
We repeat the same argument.
\\
Similarly, $M'\not\cong \mathcal{P}$; otherwise
$\dps \frac{b_3'}{a_3'}=\frac1 3$, so $a_3\dots a_n=3$, with $n\geq 4$;
a contradiction.
Then, there exist 
positive integers $m$ and $\alpha$ such that  $\alpha<m$ and~:
\begin{itemize}
\item[(i)]
$\displaystyle\frac{b_1}{a_1}<\frac{m-\alpha}{m}$;\\
\item[(ii)]
$\displaystyle\frac{b_3'}{a_3'}< \frac{\alpha}{m}$; and\\
\item[(iii)]
$\displaystyle\frac{b_2}{a_2}<\frac{1}{m}$.
\end{itemize}
Since $b_1/a_1>b_2/a_2> \dots>b_n/a_n$, we obtain that $\dps\frac{b_i}{a_i}<\frac{1}{m}$ for $i\in\{2,3,\dots,n\}$, which implies that 
$m$ and $\alpha$ can be chosen so that they 
satisfy Property $(*)$ for 
the rational coefficients $\dps \frac{b_i}{a_i}$ (for $i\in\1n$).
\subsection{Step 2~: General results for $n=3$}
\quad
\\
First, note that if $m$ and $\alpha$ are positive integers such that $\alpha<m$,
which satisfy Property $(*)$ then, by definition of $\alpha_1$ and $\alpha_2$~: 
$(i)$ and $(ii)$ of Property $(*)$ are respectively 
equivalent to $(I)$ and $(II)$ bellow.
$$
\left\{ \begin{array}{lll}
(I)\qquad \alpha<m\alpha_1\\
\\
(II) \qquad m\alpha_2<\alpha\\
\end{array} \right.
$$
Let 
$$
\fbox{
\vbox{\hsize5.5truecm
\begin{center}
$a=a_1a_2$ and
$b=a-b_1a_2-b_2a_1$;

\medskip
then $\dps\frac{b}{a}=\alpha_1-\alpha_2$.
\end{center}
}}
$$
\medskip\noindent
Let $[.]$ denote the integral value,

i.e. $[x]$ is the integer $k$ such that $k\leq x< k+1$, for all real $x$.
\\
Let $N=[a/b]$, hence 
$$
\fbox{
\vbox{\hsize3truecm
\begin{center}
$\dps N=\bigg[\frac{1}{\alpha_1-\alpha_2}\bigg]$.
\end{center}
}}
$$
\begin{LM}\label{m et alpha}
Recall that $\alpha$ and $m$ are integers.
The two following properties are satisfied.
\begin{itemize}
\item[(i)] $N\geq 4$;

\item[(ii)]
If $N\alpha_1-1\leq\alpha\leq N\alpha_1$ and $N-1\leq m$,
then $0<\alpha<m$.
\end{itemize}
\end{LM}
\begin{proof}
{\it Proof of $(i)$}.
By $(E2)$ and $(5)$, $\alpha_1-\alpha_2<\dps\frac 1 4-\frac{\varepsilon}{aa_3}$,
i.e. $\dps 4b<a-\frac{4\varepsilon}{a_3}$.

Note that $(5)$ implies that $a_3\geq 5$ ($b_3\geq 1$).

Then (since $a$ and $b$ are positive integers) $4b\leq a$.
So $N=\dps\bigg[\frac a b\bigg]\geq 4$.
\\\\
{\it Proof of $(ii)$}.
Let $\alpha$ and $m$ such that
$N\alpha_1-1\leq\alpha\leq N\alpha_1$ and $N-1\leq m$.

Now, we can check that $0<\alpha<m$.

The fact that $\alpha<m$ is trivial because $\alpha_1\leq 1/2$.\\
Let's check that $\alpha\geq1$.
\\
First, note that if $b=1$ then
$N\alpha_1-1=\dps a\frac{(a_1-b_1)}{a_1}-1=a_2(a_1-b_1)-1=b_2a_1>1$.
\\
Then, we assume $b>1$.
\\
We proceed by contradiction. Assume $\alpha=0$, then $N\alpha_1\leq1$.
\\

$N\alpha_1\leq1\Leftrightarrow 
\dps \alpha_1\leq \frac{1}{N}$, which is $\dps\frac{1}{[a/b]}$.
\\

Hence, 
$(E_2)$ implies 
$\dps\frac{b_2}{a_2}+\frac{b_3}{a_3}\leq\frac{1}{[a/b]}+\frac{\epsilon}{a_1a_2a_3}$.
\\

Since $\dps\frac{b_3}{a_3}<\frac{b_2}{a_2}$, 
$\dps\frac{b_3}{a_3}\leq\frac{1}{2[a/b]}+\frac{\epsilon}{2a_1a_2a_3}$
and so 
$\dps 2b_3[a/b]\leq a_3+\frac{\epsilon[a/b]}{a}$.
\\

Now, $b>1$ implies $\dps\frac{[a/b]}{a}< 1$ hence $2b_3[a/b]\leq a_3$.
\\

Furthermore 
$\dps [a/b]>a/b-1\Rightarrow \frac{a}{b}-1<\frac{a_3}{2b_3}$
and so~:
$\dps ab_3-bb_3<\frac{a_3b}{2}$.
\\

Then $\dps ab_3-\frac{a_3b}{2}< bb_3$.
\\

Finally, note that $(E_2)\Leftrightarrow ab_3-a_3b=\epsilon$, i.e.
$\dps ab_3-\frac{a_3b}{2}=\epsilon+\frac{a_3b}{2}$.
\\

Hence $\dps\epsilon+\frac{a_3b}{2}< bb_3\Leftrightarrow \frac{b_3}{a_3}>\frac{1}{2}+\frac{\epsilon}{ba_3}$.
\\

By $(5)$ $\varepsilon=-1$ and
 $\dps\frac{1}{4}>\frac{1}{2}+\frac{-1}{ba_3}$, i.e. $\dps\frac{1}{ba_3}>\frac{1}{4}$,
 so $ba_3<4$.
 \\
 
This  is a contradiction because $(5)$ implies that $a_3\geq5$ and $b\geq 2$.
\end{proof}
\begin{LM}\label{(I) et (II)}
Let $r=N\alpha_1-[N\alpha_1]$,
$r'=a/b-[a/b]$ and $r''=a\alpha_1/b-[a\alpha_1/b]$.
\\
If $N\alpha_1\in\mathbb N$, let $(\alpha,m)=(N\alpha_1-1,N-1)$.
\\
If $N\alpha_1\not\in\mathbb N$ and $r'\alpha_2\leq r''<\alpha_1r'$,
let $(\alpha,m)=([N\alpha_1],N)$.
\\
Otherwise,
let $(\alpha,m)=([N\alpha_1],N-1)$.

Then the integers $m$ and $\alpha$ are positive integers which satisfy $(I)$ and $(II)$
and $\alpha<m$.
\end{LM}
The proof of this lemma is the main part of Step 3,
but does not depend on $\varepsilon=\pm1$.
The fact that $0<\alpha<m$ is an immediate consequence of Lemma~\ref{m et alpha}.
\subsection{Step 3~: $n=3$ and $\epsilon=-1$}
\quad
\\\
Let us consider {\it Property $(**)$} bellow~:
 $$
(**) \left\{ \begin{array}{ll}
(I)& \alpha<m\alpha_1\\
\\
(II)& m\alpha_2<\alpha \\
\\
(III)&\displaystyle\frac{b}{a}<\frac{1}{m}
\end{array} \right.
$$
By $(E2)~:\ \dps \epsilon=-1\Rightarrow \frac{b_3}{a_3}<\frac{b}{a}$, then
Property $(**)$ implies trivially Property $(*)$,
i.e. if there exist positive integers $m$ and $\alpha$, such that $\alpha<m$ which satisfy Property $(**)$, then they satisfy Property $(*)$.
\\
We will separate the cases where
$N\alpha_1\in\mathbb N$ or $N\alpha_1\not\in\mathbb N$.
If $N\alpha_1\not\in\mathbb N$, let

$$
\fbox{
\vbox{\hsize10.1truecm
\noindent
$
\begin{array}{c}
r=N\alpha_1-[N\alpha_1]\qquad
r'=a/b-[a/b] \qquad
r''=a\alpha_1/b-[a\alpha_1/b]
\end{array}
$
}}
$$

\medskip
\begin{CL}\label{N et alpha1}
$\dps N\alpha_1=\frac{\alpha_1}{\alpha_1-\alpha_2}-\alpha_1r'$.
\end{CL}
\begin{proof}
By definition of $r'$,
$\dps N\alpha_1=[a/b]\alpha_1=(a/b-r')\alpha_1$.
\\
Then
$\dps N\alpha_1=\frac{\alpha_1}{\alpha_1-\alpha_2}-\alpha_1r'$.
\end{proof}
\begin{CL}\label{r etc}
$\dps N\alpha_1=
\big[\frac{\alpha_1}{\alpha_1-\alpha_2}\big]+r''-\alpha_1r'$.
\end{CL}
\begin{proof}
By Claim~\ref{N et alpha1}
$\dps\frac{\alpha_1}{\alpha_1-\alpha_2}-\alpha_1r'=N\alpha_1$.
\\

Moreover,
$\dps \frac{\alpha_1}{\alpha_1-\alpha_2}=\frac{a\alpha_1}b=
\big[\frac{\alpha_1}{\alpha_1-\alpha_2}\big]+r''$, by definition of $r''$.
\end{proof}
\begin{CL}\label{r''}
If $\dps [N\alpha_1]=\big[\frac{\alpha_1}{\alpha_1-\alpha_2}\big]-1$
then $r''=r+\alpha_1r'-1$.
\end{CL}
\begin{proof}
First, we may note that $\dps N\alpha_1=[N\alpha_1]+r$, by definition of $r$.
\\

Assume that $\dps [N\alpha_1]=\big[\frac{\alpha_1}{\alpha_1-\alpha_2}\big]-1$.
\\

By Claim~\ref{N et alpha1},
$\dps \frac{\alpha_1}{\alpha_1-\alpha_2}-\alpha_1r'=
N\alpha_1=
 [N\alpha_1]+r=
\big[\frac{\alpha_1}{\alpha_1-\alpha_2}\big]-1+r$.
\\

Hence 
$\dps\frac{\alpha_1}{\alpha_1-\alpha_2}=
\big[\frac{\alpha_1}{\alpha_1-\alpha_2}\big]+\alpha_1r'+r-1$.
\\

So $r''=r+\alpha_1r'-1$, by definition of $r''$.
\end{proof}

\medskip\noindent
We want to find positive integers 
$m$ and $\alpha$ such that $\alpha<m$,
and which satisfy Property~$(**)$.
First, we consider separately the case $b=1$.

\begin{LM}\label{case b=1}
If $b=1$ then $m=a-1$ and $\alpha=a_1b_2$ satisfy property $(*)$
and $0<\alpha<m$.
\end{LM}
\begin{proof}
Assume that $b=1$ and let $m=a-1$ and $\alpha=a_1b_2$.
First, we can check that $0<\alpha<m$ because
$a_1b_2\leq a_1(a_2-1)<a_1a_2-1$.
Now, we want to check successively $(I)$ to $(III)$.
\\\\
$(I)\Leftrightarrow \alpha<m\alpha_1$.

$\dps m\alpha_1=(a_1a_2-1)\frac{a_1-b_1}{a_1}>a_2(a_1-b_1)-1$.
\\

Since $b=1$, $a_2(a_1-b_1)-1=a_1b_2<m\alpha_1$.
\\\\
$(II)\Leftrightarrow m\alpha_2<\alpha$.

$\dps m\alpha_2=(a_1a_2-1)\frac{b_2}{a_2}=a_1b_2-\frac{b_2}{a_2}<\alpha$.
\\\\
$(III)\Leftrightarrow \dps\frac ba<\frac1m$.
\\

Since $b=1$ and $m=a-1$, $(III)$ is direct.
\end{proof}

In the following of this section, we assume that $b\not=1$.
We distinguish the three following cases.
\\

{\sc Case} A : $N\alpha_1\in\mathbb N$. Then $(\alpha,m)=(N\alpha_1-1,N-1)$.

{\sc Case} B~: $N\alpha_1\not\in\mathbb N$ and $r'\alpha_2>r''$ or $r''\geq \alpha_1r'$.
Then $(\alpha,m)=([N\alpha_1],N)$.

{\sc Case} C~:
$N\alpha_1\not\in\mathbb N$ and $r'\alpha_2\leq r''<\alpha_1r'$.
Then $(\alpha,m)=([N\alpha_1],N-1)$.
\\\\
First, we prove $(III)$ of Property $(**)$.
Then Lemma~\ref{(I) et (II)} concludes this step.
Note that, for $\varepsilon=1$, Lemmata~\ref {case b=1}~and~\ref{(I) et (II)} imply that $(I)$ to $(III)$ are true, but $(III)$ does not imply $(iii)$.

Furthermore, we may note that $b\not=1$ if and only if
$\dps\big[\frac ab\big]<\frac ab$ because $a$ and $b$ are positive coprime integers
(since $a_1$ and $a_2$ are so).
\begin{LM}\label{III}
We assume that $b\not=1$.
If the integers $\alpha$ and $m$ 
are chosen as in Lemma~\ref{(I) et (II)}
(according to Cases A, B or C)
then $\dps\frac b a<\frac 1 m$.
\end{LM}

\begin{proof}
Let $\alpha$ and $m$ be integers as in Cases A, B and C successively.
\\\\
Assume that Case A or Case C is satisfied.

Then $m=N-1$. Therefore $(III)$ is trivial, because $m=N-1=[a/b]-1<a/b$, so $1/m>b/a$.
\\\\
Assume now that Case B is satisfied.

Then 
$(III) \Leftrightarrow b/a<1/N \Leftrightarrow N<a/b$, which is satisfied because $N=[a/b]$
and $b\not=1$.
\end{proof}

\medskip\noindent
\begin{center}
{\bf  Proof of Lemma~\ref{(I) et (II)}}
\end{center}
We may recall that the proof does not depend on $\varepsilon=\pm 1$.

\medskip\noindent
We only have to show that the considered integers in
Cases A, B and C satisfy $(I)$ and $(II)$.
We may recall that $0<\alpha<m$ by Lemma~\ref{m et alpha}.

$$
\fbox{
\vbox{\hsize7.7truecm
\noindent
{\sc Case $A$~: $N\alpha_1\in\mathbb N$, $(\alpha,m)=(N\alpha_1-1, N-1)$}
}}$$

\smallskip\noindent
$(I)\Leftrightarrow \alpha<m\alpha_1$.

So, $(I) \Leftrightarrow N\alpha_1-1<(N-1)\alpha_1\Leftrightarrow \alpha_1<1$ which is true because $\dps 0<\frac{b_1}{a_1}<1$.
\\\\
$(II)\Leftrightarrow m\alpha_2<\alpha$.

$(II)\Leftrightarrow (N-1)\alpha_2<N\alpha_1-1\Leftrightarrow 1-\alpha_2<N(\alpha_1-\alpha_2)$.
\\

Therefore,
$(II)\dps\Leftrightarrow \frac{1}{\alpha_1-\alpha_2}-N<\frac{\alpha_2}{\alpha_1-\alpha_2}$.
\\

But recall that $\dps N=\big[\frac{1}{\alpha_1-\alpha_2}\big]$, hence 
$\dps\frac{1}{\alpha_1-\alpha_2}-N<1$.
Thus, $(II)$ follows from Claim~\ref{n=3}~$(6)$.
\\

$$
\fbox{
\vbox{\hsize8.7truecm
\noindent
{\sc Case $B$~: $r''\geq\alpha_1r'$ or $r''<r'\alpha_2$, 
$(\alpha,m)=([N\alpha_1],N)$}
}}$$

\smallskip\noindent
$(I)\Leftrightarrow \alpha<m\alpha_1$.

$(I)$ is trivially satisfied~:
$(I) \Leftrightarrow [N\alpha_1]<N\alpha_1$.
\\\\
$(II)\Leftrightarrow m\alpha_2<\alpha$.

$(II)\Leftrightarrow N\alpha_2<[N\alpha_1]
\Leftrightarrow N\alpha_2<N\alpha_1-r$. 
\\

Then
$(II)\Leftrightarrow r<N(\alpha_1-\alpha_2) \Leftrightarrow
r<(a/b-r')(\alpha_1-\alpha_2)$,
by definition of $r'$.
\\

Recall that $b/a=\alpha_1-\alpha_2$, so
\\

$(II)\Leftrightarrow r<1-r'(\alpha_1-\alpha_2)\Leftrightarrow r+r'(\alpha_1-\alpha_2)<1$.
\\\\
We want to prove that $r+r'(\alpha_1-\alpha_2)<1$.
\\\\
Assume first, that $r''\geq\alpha_1r'$.

Then Claim~\ref{r etc} implies that~:
$\dps [N\alpha_1]=\big[\frac{\alpha_1}{\alpha_1-\alpha_2}\big]$.
\\

By Claim~\ref{N et alpha1}
$\dps \frac{\alpha_1}{\alpha_1-\alpha_2}-\alpha_1r'=[N\alpha_1]+r$, then
$\dps [N\alpha_1]=\frac{\alpha_1}{\alpha_1-\alpha_2}-\alpha_1r'-r$.
\\

Thus 
$\dps \frac{\alpha_1}{\alpha_1-\alpha_2}=
\big[\frac{\alpha_1}{\alpha_1-\alpha_2}\big]
+\alpha_1r'+r$; so $r''=\alpha_1r'+r<1$.
\\

Now, we can see that~: $r+r'(\alpha_1-\alpha_2)<r+\alpha_1r'<1$,
which proves $(II)$.
\\\\
Now, we may assume that $r''<r'\alpha_2$.

By the previous work, we may assume that $r''<\alpha_1r'$.

Then Claim~\ref{r etc} implies that
$\dps [N\alpha_1]=\big[\frac{\alpha_1}{\alpha_1-\alpha_2}\big]-1$.

Therefore, Claim~\ref{r''} implies that $r''=r+\alpha_1r'-1$.

Recall that we want to show that $r+r'(\alpha_1-\alpha_2)<1$.

Since $r''<r'\alpha_2$,
we obtain~:

$r+r'(\alpha_1-\alpha_2)=r+\alpha_1r'-r'\alpha_2<r+\alpha_1r'-r''$.

Here, $r+\alpha_1r'-r''=1$, which gives the required inequality.

$$
\fbox{
\vbox{\hsize9.5truecm
\noindent
{\sc Case $C$~: $r''<\alpha_1r'$ and $r''\geq r'\alpha_2$,
$(\alpha,m)=([N\alpha_1],N-1)$}
}}$$

\smallskip\noindent
$(I)\Leftrightarrow \alpha<m\alpha_1$.

$(I)\Leftrightarrow [N\alpha_1]<(N-1)\alpha_1\Leftrightarrow \alpha_1<r$.
\\

Since $r''<\alpha_1r'$, 
by Claim~\ref{r etc}~:
$\dps [N\alpha_1]=\big[\frac{\alpha_1}{\alpha_1-\alpha_2}\big]-1$.
\\

Then, by Claim~\ref{r''}~: $r''=r+\alpha_1r'-1$.
\\

Thus $(I)\Leftrightarrow \alpha_1<r''-\alpha_1r'+1 \Leftrightarrow
\alpha_1r'-r''<1-\alpha_1$.
\\

Hence, $\dps(I)\Leftrightarrow \alpha_1r'-r''<\frac{b_1}{a_1}$,
because $\dps 1-\alpha_1=\frac{b_1}{a_1}$.
\\

On the other hand,
$\dps\alpha_1r'-r''< \alpha_1-r''$ and $\dps \alpha_1-r''\leq\alpha_1\leq\frac{b_1}{a_1}$
by (3).

Therefore $(I)$ is satisfied.
\\\\
$(II)\Leftrightarrow m\alpha_2<\alpha$.

$(II)\Leftrightarrow (N-1)\alpha_2<[N\alpha_1]$.
\\

By Claim~\ref{r etc} and the definition of $r''$, and since $r''<\alpha_1r'$~:
\\

$\dps[N\alpha_1]=[\frac{\alpha_1}{\alpha_1-\alpha_2}]-1=
\frac{\alpha_1}{\alpha_1-\alpha_2}-r''-1$.
\\

Moreover, by the definition of $r'$~:
$\dps N\alpha_2=\frac{\alpha_2}{\alpha_1-\alpha_2}-\alpha_2r'$.
\\

Hence 
$(II)\Leftrightarrow N\alpha_2<[N\alpha_1]+\alpha_2
\Leftrightarrow\dps\frac{\alpha_2}{\alpha_1-\alpha_2}-\alpha_2r'
<\frac{\alpha_1}{\alpha_1-\alpha_2}-r''-1+\alpha_2$.
\\

Therefore,
$(II)\Leftrightarrow 
r''-r'\alpha_2<\alpha_2$.
\\

On the other hand, $\alpha_2>\alpha_1-\alpha_2$ by $(6)$
and $r'\alpha_2\leq r''<\alpha_1r'$.
\\

Then,
$r''-r'\alpha_2<r'(\alpha_1-\alpha_2)<r'\alpha_2<\alpha_2$,
which proves that $(II)$ is satisfied.

\hfill{\it Proof of Lemma~\ref{(I) et (II)}}

\qed
\\\\
In conclusion, Lemma~\ref{case b=1} solve the case $b=1$.
If $b\not=1$, then for the $\alpha$ and $m$ chosen as in
Lemma~\ref{m et alpha}, we get that $0<\alpha<m$
and Lemmata~\ref{(I) et (II)}~together with~\ref{III} show that 
they satisfy $(I), (II)$ and $(III)$.
Therefore, Property~$(*)$ is satisfied for $n=3$ and $\varepsilon=-1$.
\subsection{Step 4~: $n=3$ and $\epsilon=1$}
\quad
\\
Recall that $a=a_1a_2$ and
$b=a-b_1a_2-b_2a_1$.

We assume that $\epsilon=1$ then $(E2)$ gives~:
$$(E3)\qquad \dps\frac{b_3}{a_3}= \frac{b}{a}+\frac{1}{a_1a_2a_3}$$
so~:
$$(7)\qquad ab_3-ba_3= 1.$$
Then (Bezout relation) there exists a unique pair of positive coprime integers $(u,v)$ such that~:
$$(8)\left\{ \begin{array}{lll}
au-bv=1;\\
0<u\leq b&\hbox{and}\\
0<v\leq a\\
\end{array} \right.
$$
Now, $(7)$ implies that there exists $p\in\mathbb N$ such that
$$\left\{ \begin{array}{ll}
b_3=u+bp&\hbox{and}\\
a_3=v+ap\\
\end{array} \right.
$$
Moreover, for all $p\in\mathbb N$, we have~:
$$\dps(9)\qquad \frac{u}{v}\geq \frac{u+bp}{v+ap}>\frac{u+b(p+1)}{v+a(p+1)}>\frac{b}{a}$$
We want to find positive integers $\alpha$ and $m$ 
such that $\alpha<m$ and satisfy Pro\-perty~$(*)$.
We consider separately the three following cases.
\\

{\sc Case I}~: 
 $u\not=1$.

{\sc Case II} : 
 $u=1$ and $b=1$.

{\sc Case III} : 
 $u=1$ and $b\not=1$.
 \\
$$
\fbox{
\vbox{\hsize2.5truecm
\noindent
{\sc Case I~: $u\not=1$}
}}$$
\\
We will choose the integers $\alpha$ and $m$ as in Lemma~\ref{(I) et (II)},
so $m\in\{N-1,N\}$.
By $(9)$, if $\dps \frac u v<\frac 1 m$, then $(iii)$ of
Property $(*)$ is satisfied.
Therefore, Lemma~\ref{(I) et (II)} and the following lemma conclude Case I.
\begin{LM}\label{caseI}
If $N-1\leq m\leq N$ and
$u\not=1$, then $\dps\frac{u}{v}< \frac{1}{m}$.
\end{LM}
\begin{proof}
Assume that $N-1\leq m\leq N$ and $u\not=1$.
First, note that $b\not=1$, because $0<u\leq b$ implies that if $b=1$ then $u=1$.

We make the proof by contradiction.
So, we suppose that $\dps\frac{u}{v}\geq \frac{1}{m}$, and we look for a contradiction.
Note that $v=um$ cannot happen, because $u,v$ are coprime integers,
and $u$ and $m$ are at least $2$, by Lemma~\ref{m et alpha}.
Thus
$\dps\frac{u}{v}> \frac{1}{m}$.

Moreover, by Lemma~\ref{III}~: $\dps\frac b a<\frac 1 m$.
Then
$$\frac{b}{a}<\frac{1}{m}<\frac{u}{v}$$
By $(8)$~: 
$$\frac{a}{b}-\frac{v}{u}=\frac{1}{ub}$$
we obtain
$$\dps 0<m-\frac{v}{u}<\frac{a}{b}-\frac{v}{u}=\frac{1}{ub}<1$$
which implies that
$$\big[\frac{v}{u}\big]=m-1$$
Now, let 
\begin{center}
$\dps r'=\frac{a}{b}-\big[\frac{a}{b} \big]<1$, and $\dps \rho=\frac{v}{u}-\big[\frac{v}{u} \big]<1$
\end{center}
We consider separately the cases $m=N$ and $m=N-1$.
\\\\
First, assume that $\dps m=N= \big[\frac{a}{b} \big]$.

Then
$\dps\big[\frac{v}{u} \big]= \big[\frac{a}{b} \big]-1\Leftrightarrow \frac{a}{b}-r'-1=
\frac{v}{u}-\rho$;

hence 
$\dps\frac{1}{ub}=1+r'-\rho\Rightarrow 1+r'-\rho<\frac{1}{b}$ because $u\neq 1$.

Thus $\dps r'<\frac{1}{b}$ because $\rho<1$.
\\

Nevertheless, 
$\dps r'=\frac{a}{b}-\big[\frac{a}{b} \big]$, $a$ and $b$ are coprime, and $a>b$. 
\\

Hence $a=bk+l$, where $k\in\mathbb{N^*}$, and $1\leq l\leq b-1$;
\\

so $r'$ can be written 
$\dps r'=k+\frac{l}{b}-\big[k+\frac{l}{b} \big]=\frac{l}{b}\Rightarrow r'\geq \frac{1}{b}$;
which is a contradiction.
\\\\
Now, assume that $\dps m=N-1= \big[\frac{a}{b} \big]-1$.

Then 
$\dps\big[\frac{v}{u} \big]= \big[\frac{a}{b} \big]-2\Leftrightarrow \frac{a}{b}-r'-2=
\frac{v}{u}-\rho$;
\\

hence $\dps\frac{1}{ub}=2+r'-\rho$.
\\

This implies that $\dps\frac{1}{ub}>1$ because $r'$ and $\rho$ lies in $[0,1[$.

On the other hand,
$\dps\frac{1}{ub}\leq\frac 1 4$, because $b\geq2$ and $u\geq 2$.
\\

These are in a contradiction.

\hfill\textit{(Lemma~\ref{caseI})}

\end{proof}
$$
\fbox{
\vbox{\hsize4.5truecm
\noindent
{\sc Case  II~: $u=1$ and $b=1$}
}}$$
\\
We assume that $u=b=1$.
Then $au-bv=1$ gives $v=a-1$. 

We consider separately the cases where $\dps\frac{b_1}{a_1}>\frac{1}{2}$
or $\dps\frac{b_1}{a_1}=\frac{1}{2}$.
\\\\
First, assume that $\dps\frac{b_1}{a_1}>\frac{1}{2}$.

Let $m=a-2$, and $\alpha=a_2(a_1-b_1)-1=a-a_2b_1-1=b_2a_1$ (because $b=a-a_2b_1-a_1b_2=1$.
Then $0<\alpha<m$.
\\\\
We want to check $(I), (II)$ and $(iii)$.
\\
$(I)\dps\Leftrightarrow \alpha<m\alpha_1 \Leftrightarrow a_2(a_1-b_1)-1<(a_1-b_1)a_2-2\alpha_1
\Leftrightarrow\frac{1}{2}<\frac{b_1}{a_1}$ (which is satisfied here).
\\\\
$(II)\dps\Leftrightarrow m\alpha_2<\alpha \Leftrightarrow (a-2)\alpha_2<a_2(a_1-b_1)-1
\Leftrightarrow 1-2\alpha_2<a_2(a_1-b_1)-a\alpha_2$;
which is satisfied because
$a_2(a_1-b_1)-a\alpha_2=a_2(a_1-b_1)-a_1b_2=b=1$.
\\\\
By $(9)$, $(iii)$ is satisfied if $\dps \frac u v<\frac 1 m$;
which is true because $\dps \frac u v=\frac 1{a-1}$ and $\dps \frac 1 m=\frac 1{a-2}$.
\\\\\\
Now, assume that $\dps\frac{b_1}{a_1}=\frac{1}{2}$.

Then $a_1=2$ and $b_1=1$.
Since $1=b=a_2(a_1-b_1)-a_1b_2$, $a_2=1+2b_2$.
So~:
$$\dps\frac{b_2}{a_2}=\frac{b_2}{1+2b_2}$$

and by $(E2)$~:
$\dps\frac{b_3}{a_3}=\frac 1 2-\frac{b_2}{1+2b_2}+\frac 1{2(2b_2+1)a_3}$.
\\

Thus,
$\dps\frac{b_3}{a_3}=\frac{(2b_2+1)a_3-2b_2a_3+1}{2(2b_2+1)a_3}$, i.e.
$$\dps\frac{b_3}{a_3}=\frac{a_3+1}{2(2b_2+1)a_3}$$
We consider separately the cases $b_2=1$ and $b_2>1$.
\\
Assume first $b_2=1$.

Then $a_2=3$ so $\dps\frac{b_3}{a_3}=\frac{a_3+1}{6a_3}$.
\\

Therefore, we can check easily that
$\alpha=2$ and $m=5$ satisfy Property $(*)$.

$(i)\qquad \dps\frac{b_1}{a_1}=\frac 1 2<\frac{m-\alpha}{m}$, which is $\dps\frac 3 5$;
\\

$(ii)\qquad \dps\frac{b_2}{a_2}=\frac 1 3<\frac{\alpha}{m}$, which is $\dps\frac 2 5$; and
\\

$(iii)\qquad \dps\frac{b_3}{a_3}=\frac{a_3+1}{6a_3}<\frac{1}{m}$
(which is $\dps\frac 1 5$)
if and only if $a_3> 5$.
\\

By $(5)\ a_3\geq 5$, but if $a_3=5$ then $M\cong\mathcal{P}$, so $a_3>5$.
\\\\
Now, we assume that $b_2\geq 2$.

Let $\alpha=2b_2-1$ and $m=4b_2-1$.

Since $b_2\geq 2$~: $0<\alpha<m$.
We want to check $(i), (ii)$ and $(iii)$.
\\

$(i)\qquad \dps\frac{b_1}{a_1}=\frac 1 2<\frac{m-\alpha}{m}$, which is 
$\dps\frac{2b_2}{4b_2-1}$; so $(i)$ is satisfied.
\\

$(ii)\qquad \dps\frac{b_2}{a_2}=\frac{b_2}{2b_2+1}
<\frac{\alpha}{m}$, which is $\dps\frac{2b_2-1}{4b_2-1}$;
\\

and
$\dps\frac{b_2}{2b_2+1}<\frac{2b_2-1}{4b_2-1}$
if and only if $4b_2^2-b_2<4b_2^2-1$, i.e. $b_2>1$;
\\

so $(ii)$ is satisfied.
\\

$(iii)\qquad \dps\frac{b_3}{a_3}=
\frac{a_3+1}{2(2b_2+1)a_3}
<\frac{1}{m}$, which is $\dps\frac 1 {4b_2-1}$.
\\

Then, $(iii)$ is satisfied if and only if~:

$(a_3+1)(4b_2-1)<(4b_2+2)a_3$
i.e.
$4b_2<3a_3+1$.
\\

Since $\dps\frac{b_3}{a_3}=\frac{a_3+1}{2(2b_2+1)a_3}$,
$\dps b_3=\frac{a_3+1}{2(2b_2+1)}\geq 1$ (because $b_3$ is a positive integer).
\\

So $a_3+1\geq 4b_2+2$; thus $(iii)$ is satisfied.
\\
$$
\fbox{
\vbox{\hsize4.5truecm
\noindent
{\sc Case  III~: $u=1$ and $b\not=1$}
}}$$
\\
We assume $u=1$ and $b\geq 2$.
Then $a-bv=1$ by $(8)$.

\begin{CL}\label{alpha2 and v}
If $\dps \frac{b_2}{a_2}<\frac1v$ then $m=v$ and $\alpha=1$ satisfy Property $(*)$.
\end{CL}
\begin{proof}
Assume that $\dps \frac{b_2}{a_2}<\frac1v$.
To prove that $m=v$ and $\alpha=1$ satisfy Property $(*)$,
it remains to prove that 
$\dps \frac{b_1}{a_1}<\frac{v-1}v$.
Indeed, $(II)$ and $(iii)$ are trivially satisfied because
$\dps\frac{b_3}{a_3}<\frac{b_2}{a_2}<\frac1v$.
\\

By $(8)$~:
$\dps 1+\frac1{av}=\frac{b_1}{a_1}+\frac{b_2}{a_2}+\frac1v$.

But $\dps \frac{b_2}{a_2}>\frac1{av}$,
otherwise $\dps{b_2}<\frac1{a_1v}$; which is impossible.

Therefore,
$\dps \frac{b_1}{a_1}=1+\frac1{av}-\frac{b_2}{a_2}-\frac1v
<1-\frac1v$, so $\dps \frac{b_1}{a_1}<\frac{v-1}v$.
\end{proof}

Hence, in the following, we assume that $\dps\frac{b_2}{a_2}>\frac1v$
(note that the equality is impossible because the integers are coprime).
\\
\\
Let $\alpha$ be the integer such that 
$(v-1)\alpha_1-1\leq\alpha<(v-1)\alpha_1$,
and $m=min(v-1,M)$,
where $M$ is the positive integer such that
$\dps\frac{\alpha}{\alpha_2}-1\leq M<\frac{\alpha}{\alpha_2}$.
Then~:

$$
\fbox{
\vbox{\hsize6.1truecm
\noindent
$
\begin{array}{c}
\alpha=(v-1)\alpha_1-r, \hbox{\ where\ } 0<r\leq1
\\\\
\dps M=\frac{\alpha}{\alpha_2}-r', \hbox{\ where\ } 0<r'\leq1
\\\\
\hbox{and\ } m=\min(M,v-1).
\\
\end{array}
$
}}
$$

\medskip
\noindent
First, we will check that $m>\alpha>0$, then we will show that
the integers $m$ and $\alpha$ satisfy Property~$(*)$.
\begin{CL}\label{verif}
The integers $m$ and $\alpha$ satisfy~: $1\leq\alpha<m$
\end{CL}
\begin{proof}
First, we check that $\alpha\geq1$, where $\alpha=(v-1)\alpha_1-r$, $0<r\leq1$.

We show that $(v-1)\alpha_1>1$, then $\alpha>0$. Since $\alpha\in\mathbb{N}$,
$\alpha\geq 1$.
\\

By $(8)$~:
$\dps \alpha_1=\frac{b_2}{a_2}+\frac1 v-\frac{1}{a_1a_2v}$.
\\

Since $\displaystyle \frac{b_2}{a_2}>\frac{1}{v}$,
$\dps \alpha_1>\frac2 v-\frac{1}{a_1a_2v}$, i.e.
$\dps \alpha_1>\frac{2a_1a_2-1}{a_1a_2v}$.
\\

Therefore $\dps v>\frac{2a_1a_2-1}{a_1a_2\alpha_1}$,
so $\dps (v-1)\alpha_1>\frac{a_1a_2(2-\alpha_1)-1}{a_1a_2}$.
\\

Finally, recall that $\dps 1-\alpha_1=\frac{b_1}{a_1}$.
\\

Thus,
$\dps (v-1)\alpha_1>\frac{a_1a_2(1+\frac{b_1}{a_1})-1}{a_1a_2}$, i.e.
$\dps (v-1)\alpha_1>\frac{a_1a_2+a_2b_1-1}{a_1a_2}$.
\\

Since $a_2\geq 3$, $\dps (v-1)\alpha_1>\frac{a_1a_2+2}{a_1a_2}>1$.
\\\\
Now, we check that $m>\alpha$.

If $m=v-1$, this is trivial.

So, we may assume that 
$\dps m= \frac{\alpha}{\alpha_2}-r'$, where $0<r'\leq 1$.
\\

Therefore, $\dps m={\alpha}(\frac1{\alpha_2}-\frac{r'}{\alpha})$.
\\

Since $\alpha\geq 1$, $\dps \frac{r'}{\alpha}\leq r'\leq 1$, so
$\dps m\geq{\alpha}(\frac1{\alpha_2}-1)$.
\\

Finally, $(4)$ implies that $\dps\alpha_2<\frac 1 2$ and so that $m> \alpha$.
\end{proof}
To show that $\alpha$ and $m$ satisfy Property~$(*)$,
we need the following claim.

\begin{CL}\label{1/a}
$\dps\frac{\alpha_1-\alpha_2}{\alpha_1}+ \alpha_1<1-\frac1 a$.
\end{CL}
\begin{proof}
We first consider that $\dps\frac{b_1}{a_1}=\frac1 2$.
\\

Then $\alpha_1=\dps\frac1 2$ and $a=2a_2$; so $\dps 1-\frac1 a=\frac{2a_2-1}{2a_2}$.
\\

On the other hand, 
$\dps\frac{\alpha_1-\alpha_2}{\alpha_1}+ \alpha_1
=
\dps\frac3 2-2\alpha_2=\frac{3a_2-4b_2}{2a_2}$.

Then, here~:
$\dps\frac{\alpha_1-\alpha_2}{\alpha_1}+ \alpha_1<1-\frac1 a$
if and only if
$$a_2-4b_2<-1.$$
\\
Now, $(6)$ implies $\dps\alpha_2>\frac{\alpha_1}{2}$, i.e. 
$\dps 4b_2> a_2$, so $a_2-4b_2\leq -1$.
\\

We are going to show that $a_2\not=4b_2-1$ by contradiction.

First, note that
since $b_1/a_1=1/2$,
$a=2a_2$ and $b=a_2-2b_2\not=1$.
\\

On the other hand, since $a-bv=1$, 
$\dps v=\frac{a-1}b=\frac{2a_2-1}{a_2-2b_2}$.

If $a_2=4b_2-1$, then
$\dps v=\frac{8b_2-3}{2b_2-1}$.
\\

Now,
$\dps v=4+\frac1{2b_2-1}\in\mathbb{N}$ implies that $b_2=1$,
$v=5$ and $a_2=3$.
Then $b=3-2=1$; which is a contradiction.
\\

Therefore, $a_2<4b_2-1$;
which is the required inequality.
\\\\\\
Now, we assume that $\dps\frac{b_1}{a_1}>\frac1 2$, so $2b_1-a_1>0$.
\\

Then $a_1-b_1<a_1b_2(2b_1-a_1)$,
\\

so
$(a_1-b_1)+a_1^2b_2<2a_1b_1b_2$,
\\

and
$(a_1-b_1)-2a_1b_1b_2+2a_1^2b_2<a_1^2b_2$.
Therefore~:
$$(\star)\qquad (a_1-b_1)(1+2a_1b_2)<a_1^2b_2.$$

On the other hand,
$(6)$ implies that 
$2\alpha_2>\alpha_1$, i.e.
$2a_1b_2>a_2(a_1-b_1)$.

Hence, $2a_1b_2(a_1-b_1)+(a_1-b_1)>a_2(a_1-b_1)^2+(a_1-b_1)$,
i.e.
$$(2a_1b_2+1)(a_1-b_1)>a_2(a_1-b_1)^2+(a_1-b_1).$$

Therefore, by the inequality~$(\star)$~:
$$a_2(a_1-b_1)^2+(a_1-b_1)<a_1^2b_2.$$

So $\dps\frac{a_1-b_1}{a_1^2a_2}<\frac{a_1^2b_2-a_2(a_1-b_1)^2}{a_1^2a_2}$;
i.e.
$\dps\frac{\alpha_1}a<\alpha_2-{\alpha_1^2}$.
\\

Thus, we obtain $\dps\frac 1 a<\frac{\alpha_2}{\alpha_1}-\alpha_1$,
which gives the required inequality.
\end{proof}

\medskip\noindent
Now, we will show successively that $\alpha$ and $m$ satisfy $(iii),\ (II)$ and $(I)$
of Pro\-perty~$(*)$.
\\\\
{\it - $\alpha$ and $m$ satisfy $(iii)$~: $\dps\frac{b_3}{a_3}<\frac 1 m$.}

This is trivially satisfied because by $(9)$, $\dps\frac{b_3}{a_3}\leq\frac1 v$, and $m\leq v-1$.
\\\\
{\it - $\alpha$ and $m$ satisfy $(II)$~: $m\alpha_2<\alpha$.}

Since $m\leq M$, $m\alpha_2\leq \alpha-r'\alpha_2<\alpha$ (because $r'>0$) then 
$(\alpha,m)$ trivially satisfies~$(II)$.
\\\\
{\it - $\alpha$ and $m$ satisfy $(I)$~: $\alpha<m\alpha_1$.}

Since $r>0$, $(v-1)\alpha_1-r<(v-1)\alpha_1$.
Hence,
$\alpha<m\alpha_1$ if $m= v-1$. Thus, we may assume that $m=M\leq v-2$.

So, we want to show that
$\dps(v-1)\alpha_1-r<(\frac{\alpha}{\alpha_2}-r')\alpha_1$. Now~:
\\
 $$
\begin{array}{lll}
\dps(v-1)\alpha_1-r<(\frac{\alpha}{\alpha_2}-r')\alpha_1&\Leftrightarrow& 
\dps v-\frac r{\alpha_1}<\frac{(v-1)\alpha_1-r}{\alpha_2}-r'+1
\\\\
& \Leftrightarrow&
\dps v\alpha_1\alpha_2-r\alpha_2<v\alpha_1^2-\alpha_1^2-r\alpha_1-r'\alpha_1\alpha_2+
\alpha_1\alpha_2
\\\\
& \Leftrightarrow&
\dps r(\alpha_1-\alpha_2)+r'\alpha_1\alpha_2+\alpha_1(\alpha_1-\alpha_2)<v\alpha_1(\alpha_1-\alpha_2)
\\\\
& \Leftrightarrow&
\dps v(\alpha_1-\alpha_2)>\alpha_1-\alpha_2
+r\frac{\alpha_1-\alpha_2}{\alpha_1}+r'\alpha_2
\\
\end{array}
$$

Recall that $\dps\frac b a=\alpha_1-\alpha_2$ and $a-bv=1$,
so~:
$$v(\alpha_1-\alpha_2)=\dps\frac{vb}a=\frac{a-1}a=1-\frac1 a.$$

Therefore,
$\alpha$ and $m$ satisfy $(I)$ if and only if
$$\dps 1-\frac1 a>
\alpha_1-\alpha_2
+r\frac{\alpha_1-\alpha_2}{\alpha_1}+r'\alpha_2.$$

Since $r$ and $r'$ both lie in $]0,1]$~:
\\

$\dps 
\alpha_1-\alpha_2
+r\frac{\alpha_1-\alpha_2}{\alpha_1}+r'\alpha_2
<
\alpha_1-\alpha_2
+\frac{\alpha_1-\alpha_2}{\alpha_1}+\alpha_2$,
\\

i.e.
$\dps 
\alpha_1-\alpha_2
+r\frac{\alpha_1-\alpha_2}{\alpha_1}+r'\alpha_2
<
\alpha_1
+\frac{\alpha_1-\alpha_2}{\alpha_1}
<1-\dps\frac1a$, by Claim~\ref{1/a}.
\\\\
Hence, $\alpha$ and $m$ satisfy $(I)$, which ends the proof of Proposition~\ref{main}.

\end{document}